\documentclass[11pt,oneside,reqno]{amsart}

\usepackage[a4paper, total={450pt,675pt}]{geometry}
\usepackage[OT2,T1]{fontenc}
\usepackage[utf8]{inputenc}
\usepackage{amssymb}

\usepackage[dvipsnames]{xcolor}
\usepackage[
    colorlinks=true,
    linkcolor=Maroon,
    citecolor=JungleGreen,
    urlcolor=NavyBlue]{hyperref}

\usepackage[capitalize,nameinlink]{cleveref}

\usepackage{lipsum}

\usepackage{enumitem}
\usepackage{dsfont}
\usepackage{pgf,tikz}
\usetikzlibrary{arrows}
\definecolor{ffqqqq}{rgb}{1.,0.,0.}
\definecolor{qqzzff}{rgb}{0.,0.6,1.}
\usepackage{subcaption}
\usetikzlibrary{arrows}

\makeatletter
\def\@tocline#1#2#3#4#5#6#7{\relax
  \ifnum #1>\c@tocdepth 
  \else
    \par \addpenalty\@secpenalty\addvspace{#2}%
    \begingroup \hyphenpenalty\@M
    \@ifempty{#4}{%
      \@tempdima\csname r@tocindent\number#1\endcsname\relax
    }{%
      \@tempdima#4\relax
    }%
    \parindent\z@ \leftskip#3\relax
    \advance\leftskip\@tempdima\relax
    \rightskip\@pnumwidth plus4em \parfillskip-\@pnumwidth
    #5\leavevmode\hskip-\@tempdima
      \ifcase #1
       \or\or \hskip 2em \or \hskip 2em \else \hskip 3em \fi%
      #6\nobreak\relax
    \dotfill\hbox to\@pnumwidth{\@tocpagenum{#7}}\par
    \nobreak
    \endgroup
  \fi}
\makeatother

\usepackage[
    backend=bibtex,
    style=alphabetic,
    maxbibnames=10,
    sorting=nyt,
    firstinits=true]{biblatex}

\DeclareLabelalphaTemplate{
  \labelelement{
    \field[final]{shorthand}
    \field{label}
    \field[strwidth=3]{labelname}}
  \labelelement{
    \field[strwidth=2,strside=right]{year}}
}

\DeclareFieldFormat[article,book,incollection,misc]{title}{#1}
\DeclareFieldFormat[inbook,incollection]{booktitle}{#1}
\DeclareFieldFormat[misc]{date}{\textit{preprint} (#1)}
\DeclareFieldFormat{pages}{#1}
\DeclareFieldFormat{doi}{
\href{http://www.ams.org/mathscinet-getitem?mr=#1}{#1}}

\renewbibmacro{in:}{%
  \ifentrytype{article}{}{\printtext{\bibstring{in}\intitlepunct}}}

\newtheorem{theorem}{Theorem}[section]

\newtheorem{lemma}[theorem]{Lemma}

\newtheorem{corollary}[theorem]{Corollary}

\newtheorem{remark}[theorem]{Remark}

\crefname{section}{Sect.}{section}
\numberwithin{equation}{section}

\DeclareMathOperator{\ad}{Ad}

\DeclareMathOperator{\im}{Im}
\DeclareMathOperator{\re}{Re}

\DeclareMathOperator{\conv}{CH}

\DeclareMathOperator{\supp}{supp}
\DeclareMathOperator{\hess}{Hess}

\begin{document}
\title[Wave equation on certain noncompact Riemannian symmetric spaces]{
Wave equation on certain \\ noncompact symmetric spaces}

\author{Hong-Wei ZHANG}

\begin{abstract}
In this paper, we prove sharp pointwise kernel estimates 
and dispersive properties for the linear wave equation 
on noncompact Riemannian symmetric spaces $G/K$ of any rank with $G$ complex. 
As a consequence, we deduce Strichartz inequalities for a large family of admissible pairs 
and prove global well-posedness results for the corresponding semilinear equation
with low regularity data as on hyperbolic spaces.
\end{abstract}


\makeatletter
\@namedef{subjclassname@2020}{\textnormal{2020}
    \it{Mathematics Subject Classification}}
\makeatother

\maketitle
\tableofcontents

\section{Introduction}
This paper is devoted to study the dispersive properties of the linear wave equation on noncompact symmetric space $G/K$ of any rank with $G$ complex, and their applications to nonlinear Cauchy problems. This theory is well established for the wave equation on $\mathbb{R}^{d}$ with $d \ge 3$:
\begin{align*}
\begin{cases}
\partial_{t}^2 u(t,x) - \Delta_{\mathbb{R}^{d}} u(t,x) = F(t,x), \\
u(0,x) =f(x) , \ \partial_{t}|_{t=0} u(t,x) =g(x),
\end{cases}
\end{align*}
where the solutions $u$ satisfy the Strichartz inequality:
\footnote{The symbol $\lesssim$, let us recall, means precisely that there exists a constant $0 < C < + \infty$ such that $\| \nabla_{\mathbb{R} \times \mathbb{R}^d} u \|_{L^p(I; H^{-\sigma,q}(\mathbb{R}^d))} \le C \big(  \| f \|_{H^1(\mathbb{R}^d)} + \| g \|_{L^2(\mathbb{R}^d)} + \| F \|_{L^{\tilde{p}'}(I; H^{\tilde{\sigma},\tilde{q}'}(\mathbb{R}^d))} \big)$, where $\tilde{p}'$ is the dual exponent of $\tilde{p}$, defined by the formula $\frac{1}{\tilde{p}} + \frac{1}{\tilde{p}'} = 1$, and similarly for $\tilde{q}'$.}
\begin{align*}
\| \nabla_{\mathbb{R} \times \mathbb{R}^d} u \|_{L^p(I; H^{-\sigma,q}(\mathbb{R}^d))} \lesssim \left\| f \right\|_{H^1(\mathbb{R}^d)} + \left\| g \right\|_{L^2(\mathbb{R}^d)} + \left\| F \right\|_{L^{\tilde{p}'}(I; H^{\tilde{\sigma},\tilde{q}'}(\mathbb{R}^d))},
\end{align*}
on any (possibly unbounded) interval $I \subseteq \mathbb{R}$ under the assumptions that
\begin{align*}
\textstyle
\sigma = \frac{d+1}{2} \big( \frac{1}{2} - \frac{1}{q} \big), \ \widetilde{\sigma} = \frac{d+1}{2} \big( \frac{1}{2} - \frac{1}{\tilde{q}} \big),
\end{align*}
and the couples $(p,q), (\tilde{p}, \tilde{q}) \in (2, +\infty ] \times [2, 2 \frac{d-1}{d-3})$ fulfill the admissibility conditions:
\begin{align*}
\textstyle
\frac{1}{p} = \frac{d-1}{2} \big( \frac{1}{2} - \frac{1}{q} \big), \ \frac{1}{\tilde{p}} = \frac{d-1}{2} \big( \frac{1}{2} - \frac{1}{\tilde{q}} \big).
\end{align*}
Notice that this inequality also holds at the endpoint $(2, 2 \frac{d-1}{d-3})$ when $d \ge 4$, but fails without additional assumptions when $d=3$ (see \cite{GiVe1995} and \cite{KeTa1998} for more details). These estimates serve as a tool in order to find  minimal regularity conditions on the initial data ensuring well-posedness for corresponding semilinear wave equations, which is addressed in \cite{Kap1994}, and almost fully answered in \cite{LiSo1995, GLS1997, KeTa1998, DGK2001}.\\

Given the rich Euclidean theory, several attempts have been made in order to establish Strichartz inequality for dispersive equations in other settings. We are interested in Riemannian symmetric spaces of noncompact type where relevant questions are now well answered in the rank one case, see for instance \cite{Fon1997, Tat2001, MeTa2011, MeTa2012, APV2012, AnPi2014} on hyperbolic spaces, and \cite{APV2015} on Damek-Ricci spaces. A first study of the wave equation on general symmetric spaces of higher rank was carried out in \cite{Has2011}, where some non optimal estimates were obtained under the strong smoothness ,assumption.\\

In this paper, we prove sharp estimates for the non-shifted wave equation on noncompact symmetric spaces $G/K$ of any rank with $G$ complex. In this case, the Harish-Chandra $\mathbf{c}$-function and the spherical function have elementary expressions, which allow us to analyze accurately the wave kernel. For lack of such expressions in general, our present approach is limited to the class of symmetric spaces $G/K$ with $G$ complex, and maybe to a few other cases.\\

Consider the operator $W_{t}^{\sigma} = \widetilde{D}^{-\sigma} e^{it \sqrt{- \Delta}}$ defined on the symmetric space $X=G/K$, for suitable exponents $\sigma \in \mathbb{C}$, where $\rho$ denotes the half sum of positive roots, and $\widetilde{D} = \sqrt{-\Delta - |\rho|^2 + \widetilde{\rho}^2}$ is the differential operator with a fixed constant $\widetilde{\rho}\ge| \rho |$. To avoid possible singularities, we may consider the analytic family of operators 
\begin{align*}
\widetilde{W}_{t}^{\sigma} = \frac{e^{\sigma^2}}{\Gamma (\frac{d+1}{2}- \sigma)}  \widetilde{D}^{-\sigma} e^{it \sqrt{- \Delta}},
\end{align*}
in the vertical strip $0 \le \re \sigma \le \frac{d+1}{2}$, where $\Gamma$ denotes the Gamma function, see \eqref{analytic family of operators}. We denote by $\widetilde{\omega}_{t}^{\sigma}$ its $K$-bi-invariant convolution kernel. Our first and main result is the following pointwise kernel estimate, which summarizes \cref{estimate omega 0} and \cref{estimate omega infinity} proved in \cref{section pointwise}.
\begin{theorem}[Pointwise kernel estimates]
For all $\sigma\in\mathbb{C}$ with $\re\sigma=\frac{d+1}{2}$, there exists $N\in\mathbb{N}$ such that the following estimates hold for all $x\in X$:
\begin{align*}
    |\widetilde{\omega}_{t}^{\sigma}(x)| \lesssim
    (1+|H_{x}|)^{N} e^{-\langle\rho,H_{x}\rangle}
    \begin{cases}
        |t|^{-\frac{d-1}{2}}
        \qquad & \textnormal{if}\quad 0<|t|<1,\\
        |t|^{-\frac{d}{2}}
        \qquad & \textnormal{if}\quad |t|\ge1,
    \end{cases}
\end{align*}
where $H_{x}\in\overline{\mathfrak{a}^{+}}$ denotes the middle component according to the Cartan decomposition of $x$.
\end{theorem}

\begin{remark}
These kernel estimates are sharp in time and similar results hold obviously in the easier case where $\re\sigma>\frac{d+1}{2}$. The value of $N$ will be specified in \cref{estimate omega 0} and \cref{estimate omega infinity}. However, the polynomial $(1+|H_{x}|)^{N}$ is not crucial for further computations since there is an exponential decay following.
\end{remark}

By using the interpolation, we deduce our second result.
\begin{theorem}[Dispersion property]\label{dispersive estimate}
Let $2 < q, \widetilde{q} < + \infty$ and $\sigma \ge (d+1) \max ( \frac{1}{2} - \frac{1}{q} , \frac{1}{2} - \frac{1}{\widetilde{q}})$. Then there exists a constant $C>0$ such that the following dispersive estimates hold:
\begin{align*}
\| \widetilde{W}_{t}^{\sigma} \|_{L^{\widetilde{q}'}(X) \rightarrow L^{q}(X)} \le C
\begin{cases}
|t|^{- (d-1) \max ( \frac{1}{2} - \frac{1}{q} , \frac{1}{2} - \frac{1}{\widetilde{q}})} 
\qquad & \textnormal{if}\quad 0 < |t| < 1, \\
|t|^{-\frac{d}{2}} \qquad & \textnormal{if}\quad |t| \ge 1.
\end{cases}
\end{align*}
\end{theorem}

\begin{remark}
At the endpoint $q=\widetilde{q}=2$, $t \mapsto e^{it \sqrt{-\Delta}}$ is a one-parameter group of unitary operators on $L^2(X)$. 
\end{remark}

\begin{remark}
These estimates, which are proved on real hyperbolic spaces in \cite{AnPi2014}, extend straightforwardly to all the noncompact symmetric spaces of rank $1$, and more generally to all Damek-Ricci spaces. In these cases, the large time decay is $|t|^{- \frac{3}{2}}$, where $"3"$ corresponds to the so-called dimension at infinity of $X$. In higher rank, the dimension at infinity coincides with the manifold dimension $d$ of $G/K$ if and only if $G$ is complex.
\end{remark}

\begin{remark}
In \cite{CGM2002}, M. Cowling, S. Giulini and S. Meda have described the $L^p$-$L^q$ boundedness of a semi-group of operators related to the shifted Laplacian in constant time $t=1$ on general noncompact symmetric spaces. However, sharp dispersive properties with a general $t\in \mathbb{R}^{*}$ are crucial for proving Strichartz type inequalities and for studying related PDE problems. Along the lines in \cite{CGM2002}, we may derive a sharp dispersive inequality in small time $|t| \le 1$. But as mentioned by these authors, their method does not yield good estimates in large time. \\
\end{remark}

This paper is organized as follows. After recalling some basic notations and reviewing harmonic analysis on noncompact symmetric spaces in \cref{section preliminaries}, we derive pointwise kernel estimates in \cref{section pointwise}. We prove the dispersive estimates by interpolation arguments in \cref{section dispersive}. As a consequence, we deduce in \cref{section applcations}, Strichartz inequalities for a large family of admissible pairs and obtain well-posedness results for the associated semilinear wave equation. In \cref{LSS}, we discuss similar results on a class of locally symmetric spaces.\\

\noindent\textbf{Acknowledgments.}
The author is supported by the doctoral fellowship of the University of Orléans and by the Methusalem Programme grant: Analysis and Partial Differential Equations. The present paper is part of the author’s Ph.D. thesis, supervised by Jean-Philippe Anker and Nicolas Burq, and the author would like to thank them for sharing their knowledge and experience with him. The author also thanks the referees for their valuable suggestions.

\section{Preliminaries}\label{section preliminaries}
We review in this section some elementary notations and facts about noncompact symmetric spaces. We refer to \cite{Hel1962, Hel2000} for more details. \\

Let $G$ be a complex semisimple Lie group, connected, noncompact, with finite center, and $K$ be a maximal compact subgroup of $G$. The homogeneous space $X=G/K$ is a Riemannian symmetric space of  noncompact type and dimension $d \ge 3$.
Let $\mathfrak{g} = \mathfrak{k} \oplus \mathfrak{p}$ be the Cartan decomposition of the Lie algebra of $G$. The Killing form of $\mathfrak{g}$ induces a $K$-invariant inner product on $\mathfrak{p}$, hence a $G$-invariant Riemannian metric on $X$, whose tangent space at the origin is identified with $\mathfrak{p}$.\\

Fix a maximal abelian subspace $\mathfrak{a}$ in $\mathfrak{p}$. The
rank of $X$ is the dimension $\ell$ of $\mathfrak{a}$. Let $\Sigma \subset \mathfrak{a}$ be the root system of $(  \mathfrak{g}, \mathfrak{a})$ which is reduced, and denote by $W$ the Weyl group associated to $\Sigma$. Choose a set $\Sigma^{+}$ of positive roots, let $\mathfrak{a}^{+} \subset \mathfrak{a}$ be the corresponding positive Weyl chamber and $\overline{\mathfrak{a}^{+}}$ be its closure. Notice that $d = \ell + 2| \Sigma^{+} |$ in our case.
As usual, $\rho \in \mathfrak{a}^{+}$ denotes the half sum of positive roots, counted with their multiplicities, which is given in our case by 
$\textstyle \rho =  \sum_{\alpha \in \Sigma^{+}} \alpha$. It is well known that the spectrum of the negative Laplace-Beltrami operator $- \Delta$ on $L^{2}(X)$ is the half-line $[ |\rho|^2, + \infty )$.\\

Denote by $\mathfrak{n}$ the nilpotent Lie subalgebra of $\mathfrak{g}$ associated with $\Sigma^{+}$, and by $N$ the corresponding Lie subgroup of $G$. Then we have the following two decompositions of $G$:
\begin{align*}
\begin{cases}
G= N \left( \exp \mathfrak{a} \right) K \quad &(Iwasawa), \\
G= K \left( \exp \overline{\mathfrak{a}{+}} \right) K \quad &(Cartan).
\end{cases}
\end{align*}

In the Cartan decomposition, the Haar measure on $G$ writes
\begin{align}\label{Cartan decomposition for Haar measure}
\int_G f(g) dg = const. \int_{K} dk_1 \int_{\mathfrak{a}^{+}} \delta(H) dH \int_{K} f(k_1 (\exp H) k_2 ) dk_2,
\end{align}
where \footnote{The symbol $f \asymp g$ between two nonnegative expressions means that there exist two constants $0 < c_1 \le c_2 < + \infty$ such that $c_1 g \le f \le c_2 g$.}
\begin{align*}
\delta(H)
= \prod_{\alpha \in \Sigma^{+}} \left( \sinh \langle\alpha,H\rangle  \right )^{2} 
\asymp \Big\lbrace \prod_{\alpha \in \Sigma^{+}} \frac{\left\langle {\alpha, H} \right\rangle}{1 + \left\langle {\alpha, H} \right\rangle }  \Big\rbrace^2 e^{ \left\langle {2 \rho, H} \right\rangle} \qquad \forall H \in \overline{\mathfrak{a}^{+}}.\\
\end{align*}

Denote by $\mathcal{S}(K \backslash G /K)$ the Schwartz space of $K$-bi-invariant functions on $G$. The spherical Fourier transform $\mathcal{H}$ is defined by
\begin{align*}
\mathcal{H} f (\lambda) = \int_{G} f(x) \varphi_{\lambda} (x) dx \qquad \forall \lambda \in \mathfrak{a},\ \forall f \in \mathcal{S} (K \backslash G/K).
\end{align*}
Here $\varphi_{\lambda} \in \mathcal{C}^{\infty} (K \backslash G/K)$ denotes the spherical function of index $\lambda \in \mathfrak{a}_{\mathbb{C}}$, which is a radial eigenfunction of the Laplace-Beltrami operator  $\Delta$ satisfying 
\begin{equation}\label{eigenfunction}
\begin{cases}
- \Delta \varphi_{ \lambda} (x) = \left(  |\lambda|^2 + |\rho|^2 \right) \varphi_{\lambda}(x), \\
\varphi_{\lambda}(0) =1.
\end{cases}
\end{equation}
In the noncompact case, the spherical functions are given by
\begin{align}\label{spherical function noncompact}
\varphi_{\lambda} (x) = \int_{K} dk \ e^{  \langle i \lambda + \rho , A(kx) \rangle} 
\qquad \forall  \lambda \in {\mathfrak{a}},
\end{align}
where $A(kx)$ is the unique $\mathfrak{a}$-component in the Iwasawa decomposition of $kx$. It satisfies the basic estimate
\begin{align}\label{basic spherical function estimate}
| \varphi_{\lambda} (x) | \le \varphi_{0} (x) \qquad \forall \lambda \in \mathfrak{a},\ \forall x \in X,
\end{align}
where
\begin{align}
\varphi_{0} (\exp H) 
= \prod_{\alpha \in \Sigma^{+}} \frac{\left\langle {\alpha,H} \right\rangle}{\sinh \left\langle {\alpha ,H} \right\rangle}
\asymp \Big\lbrace \prod_{\alpha \in \Sigma^{+}} (1 + \left\langle {\alpha, H} \right\rangle ) \Big\rbrace e^{- \left\langle {\rho, H} \right\rangle} 
\qquad \forall H \in  \overline{\mathfrak{a}^{+}}.
\end{align}
If $G$ is complex, we have also
\begin{align}\label{spherical function complex}
\varphi_{\lambda} (x) =  \varphi_{0}(x) \int_{K} dk \ e^{i \left\langle {(\ad k). \lambda,x} \right\rangle}.
\end{align}
{ } \\

Recall that $W$ is the Weyl group associated to $\Sigma$.
We denote by $\mathcal{S} \left( \mathfrak{a} \right)^{W}$ the subspace of $W$-invariant functions in the Schwartz space $\mathcal{S} \left( \mathfrak{a} \right)$. Then $\mathcal{H}$ is an isomorphism between $\mathcal{S}(K \backslash G /K)$ and $\mathcal{S} \left( \mathfrak{a} \right)^{W}$. The inverse spherical Fourier transform is defined by
\begin{align}\label{inverse formula}
f (x) = const. \int_{\mathfrak{a}} d \lambda \ \mathcal{H}f(\lambda) \varphi_{\lambda} (x) |\mathbf{c(\lambda)}|^{-2} 
\qquad \forall x \in G ,\ \forall f  \in \mathcal{S} (\mathfrak{a})^{W},
\end{align}
where $\mathbf{c(\lambda)}$ is the Harish-Chandra $\mathbf{c}$-function. If $G$ is complex, the Plancherel density reads 
\begin{align}\label{c function complex}
| \mathbf{c} ( \lambda)|^{-2} = \mathbf{\pi} (\lambda)^2 = {\textstyle \prod_{\alpha \in \Sigma^{+}}} \left\langle {\alpha, \lambda} \right\rangle^2,
\end{align}
and in particular, the inverse spherical Fourier transform \eqref{inverse formula} becomes
\begin{align*}
f (x) = const. \ \varphi_{0}(x) \int_{\mathfrak{p}} d \lambda \ \mathcal{H}f(\lambda)  \ e^{- i \left\langle {(\ad k). \lambda,x} \right\rangle}.\\
\end{align*}

Recall at last the asymptotic expansion of the Bessel function of the first kind: 
\begin{align}\label{Bessel function asymp}
J_{m}(r) \sim r^{-1/2} e^{ir} \sum_{j=0}^{\infty} a_j(m) r^{-j} + r^{-1/2} e^{-ir} \sum_{j=0}^{\infty} b_j(m) r^{-j}
\qquad \textnormal{as}\quad r \rightarrow + \infty,
\end{align}
where $a_j(m)$ and $b_j(m)$ are suitable coefficients, see for instance 
\cite[338]{Ste1993}.\\
\section{Pointwise estimates of the wave  kernel}\label{section pointwise}
In this section, we derive pointwise estimates for the $K$-bi-invariant convolution kernel $\omega_{t}^{\sigma}$ of the operator $W_{t}^{\sigma} = \widetilde{D}^{-\sigma} e^{it \sqrt{- \Delta}}$ on the symmetric space $X$:
\begin{align*}
W_{t}^{\sigma}  f(x) = f * \omega_{t}^{\sigma}(x) 
= \int_{G} dy \ \omega_{t}^{\sigma} (y^{-1}x) f(y),
\end{align*}
for suitable exponents $\sigma \in \mathbb{C}$. Via the spherical Fourier transform and \eqref{eigenfunction}, the negative Laplace-Beltrami operator $- \Delta$ corresponds to  $|\lambda|^2 + |\rho|^2$, hence the operator $\widetilde{D}$ to $\sqrt{|\lambda|^2 + \widetilde{\rho}^2}$. The inverse spherical Fourier transform implies that
\begin{align}\label{wave kernel}
\omega_{t}^{\sigma} (x) = const. \int_{\mathfrak{a}} \ d \lambda \ | \mathbf{c} ( \lambda)|^{-2} \ ( | \lambda |^2 + \widetilde{\rho}^2)^{ - \frac{\sigma}{2}} e^{it \sqrt{| \lambda|^2 + |\rho|^2}} \varphi_{\lambda}(x).
\end{align}
As already observed for hyperbolic spaces (see for instance \cite{APV2012}), the classical Euclidean rescaling method fails since this kernel has different behaviors depending whether $t$ is small or large.  We will prove sharp time pointwise estimates for this kernel along the lines of \cite{AnPi2014}.\\

Consider smooth even cut-off functions $\chi_{0}$ and $\chi_{\infty}$ on $\mathbb{R}$ such that
\begin{align*}
\chi_{0}(r) + \chi_{\infty} (r) =1
\qquad \textnormal{and} \qquad 
\begin{cases}
\chi_{0} (r) = 1 \qquad & \textnormal{if}\quad |r| \le 1, \\
\chi_{\infty} (r) = 1 \qquad & \textnormal{if}\quad |r| \ge 2. \\
\end{cases}
\end{align*}
Let us split up
\begin{align*}
\omega_{t}^{\sigma} (x) =&
\omega_{t}^{\sigma,0} (x) + \omega_{t}^{\sigma, \infty} (x) \\
=&
const. \int_{\mathfrak{a}} \ d \lambda \  \chi_{0}^{\rho} (\lambda) | \mathbf{c} ( \lambda)|^{-2} \ ( | \lambda |^2 + \widetilde{\rho}^2)^{ - \frac{\sigma}{2}} e^{it \sqrt{| \lambda|^2 + |\rho|^2}} \varphi_{\lambda}(x) \\
& + const. \int_{\mathfrak{a}} \ d \lambda \  \chi_{\infty}^{\rho} (\lambda) | \mathbf{c} ( \lambda)|^{-2} \ ( | \lambda |^2 + \widetilde{\rho}^2)^{ - \frac{\sigma}{2}} e^{it \sqrt{| \lambda|^2 + |\rho|^2}} \varphi_{\lambda}(x)
\end{align*}
where $ \chi_{0}^{\rho} (\lambda) = \chi_{0} (| \lambda | / | \rho |)$ and $ \chi_{\infty}^{\rho} (\lambda) = \chi_{\infty} (| \lambda | / | \rho |)$ are radial cut-off functions defined on $\mathfrak{a}$.
We shall see later that the kernel $\omega_{t}^{\sigma, \infty}(x)$ has a logarithmic singularity on the sphere $|x| = t$ when $\sigma = \frac{d+1}{2}$. To get around this problem, we consider the analytic family of operators
\begin{align}\label{analytic family of operators}
\widetilde{W}_{t}^{\sigma, \infty} = \frac{e^{\sigma^2}}{\Gamma (\frac{d+1}{2}- \sigma)} \ \chi_{\infty}^{\rho} \big( \sqrt{- \Delta - | \rho |^2} \big) \widetilde{D}^{-\sigma} e^{it \sqrt{- \Delta}},
\end{align}
in the vertical strip $0 \le \re \sigma \le \frac{d+1}{2}$ and the corresponding kernels
\begin{align*}
\widetilde{\omega}_{t}^{\sigma, \infty} (x)
=   \frac{e^{\sigma^2}}{\Gamma (\frac{d+1}{2}- \sigma)} \  \int_{\mathfrak{a}} \ d \lambda \  \chi_{\infty}^{\rho} (\lambda) | \mathbf{c} ( \lambda)|^{-2} \ ( | \lambda |^2 + \widetilde{\rho}^2)^{ - \frac{\sigma}{2}} e^{it \sqrt{| \lambda|^2 + |\rho|^2}} \varphi_{\lambda}(x).
\end{align*}
Notice that the Gamma function will allow us to deal with the boundary point $\sigma = \frac{d+1}{2}$, while the exponential function ensures boundedness at infinity in the vertical strip.

\begin{theorem}\label{estimate omega 0}
Let $\sigma \in \mathbb{C}$. The following pointwise estimates hold for the kernel $\omega_{t}^{\sigma,0}$:
\begin{enumerate}[label=(\roman*)]
\item For all $t \in \mathbb{R}^{*}$, we have
\begin{align*}
|\omega_{t}^{\sigma,0}(x) | \lesssim \varphi_{0}(x).
\end{align*}

\item Assume that $|t| \ge 1$.
\begin{enumerate}
\item If $ \frac{|x|}{|t|} \ge \frac{1}{2}$, then
\begin{align*}
|\omega_{t}^{\sigma,0}(x) | \lesssim |t|^{-N} (1+ |x|)^{N} \varphi_{0}(x),
\end{align*}
for every $N \in \mathbb{N}$.
\item If $ \frac{|x|}{|t|} \le \frac{1}{2}$, then
\begin{align*}
| \omega_{t}^{\sigma,0}(x) | \lesssim |t|^{-\frac{d}{2}} (1+|x|)^{\frac{d-\ell}{2}} \varphi_{0} (x).
\end{align*}
\end{enumerate}
\end{enumerate}
\end{theorem}

\begin{proof}
By symmetry we may assume that $t >0$. Recall that
\begin{align}\label{omega 0}
\omega_{t}^{\sigma,0} (x)
= const. \int_{\mathfrak{a}} \ d \lambda \  \chi_{0}^{\rho} (\lambda) | \mathbf{c} ( \lambda)|^{-2} \ ( | \lambda |^2 + \widetilde{\rho}^2)^{ - \frac{\sigma}{2}} e^{it \sqrt{| \lambda|^2 + |\rho|^2}} \varphi_{\lambda}(x).
\end{align}
$(i)$ follows from the representation \eqref{c function complex} and the  estimate \eqref{basic spherical function estimate} that
\begin{align*}
| \omega_{t}^{\sigma,0} (x) | 
\lesssim \int_{|\lambda| < 2 | \rho |} d\lambda \ | \lambda |^{d-\ell} | \varphi_{\lambda} (x) | \lesssim \varphi_{0} (x).
\end{align*}
Moreover $(ii).(a)$ is straightforward since $ \frac{|x|}{t} \ge \frac{1}{2}$.\\

We prove $(ii).(b)$ by substituting in \eqref{omega 0} the integral representation \eqref{spherical function noncompact} of $\varphi_{\lambda}$ and the expression \eqref{c function complex} of $ | \mathbf{c} ( \lambda)|^{-2} $. Specifically,
\begin{align*}
\omega_{t}^{\sigma,0} (x) = const. \
\int_{K} dk \ e^{ \langle \rho, A(kx)\rangle} 
\int_{\mathfrak{a}} \ d \lambda \ a(\lambda) \ e^{it \sqrt{| \lambda|^2 + |\rho|^2} + i \langle {\lambda, A (kx)} \rangle},
\end{align*}
where $a(\lambda) = \chi_{0}^{\rho}  (\lambda) \pi (\lambda)^2 ( | \lambda |^2 + \widetilde{\rho}^2)^{ - \frac{\sigma}{2}}$. Since
\begin{align*}
\int_{K} dk \ e^{ \langle \rho, A (kx) \rangle} = \varphi_{0}(x),
\end{align*}
it remains to estimate the oscillatory integral
\begin{align*}
I_{0}(t,x) = 
\int_{\mathfrak{a}} \ d \lambda \ a(\lambda) \ e^{it \psi (\lambda)}
\end{align*}
with amplitude $a(\lambda)$ and phase

\begin{align}\label{Sect 3. phase}
\textstyle
\psi (\lambda ) = \sqrt{| \lambda|^2 + |\rho|^2} + \big\langle { \frac{A}{t}, \lambda} \big\rangle,
\end{align}
where $A = A(kx)$ is a vector in $\mathfrak{a}$ such that $|A(kx)| \le |x|$. This is achieved by carrying out in our case, the stationary phase analysis described in \cite[\nopp VIII.2.3]{Ste1993}. The critical point $\lambda_{0}$ of the phase $\psi$ is given by
\begin{align*}
\textstyle
( | \lambda_{0} |^2 + | \rho |^2)^{-\frac{1}{2}} \lambda_{0} = 
- \frac{A}{t}.
\end{align*}
Hence
\begin{align}\label{critical point}
\textstyle
| \lambda_{0} | = | \rho | \frac{|A|}{t} \big( 1 - \frac{|A|^2}{t^2} \big)^{-\frac{1}{2}}
\le | \rho | \frac{|x|}{t} \big( 1 - \frac{|x|^2}{t^2} \big)^{-\frac{1}{2}}.
\end{align}
Denote by $B(\lambda_{0} , \eta)$ the ball in $\mathfrak{a}$  centered at $\lambda_{0}$ :
\begin{align*}
B(\lambda_{0} , \eta) = \big\lbrace \lambda \in \mathfrak{a} \ \big| \ | \lambda - \lambda_{0} | \le \eta \big\rbrace,
\end{align*}
where the radius $ \eta$ will be chosen later. Notice that if $\frac{|x|}{t} \le \frac{1}{2}$, then $|  \lambda_{0} | < \frac{|\rho|}{\sqrt{3}}$ and $| \lambda | < | \rho | + \eta$ for all $\lambda \in B(\lambda_{0} , \eta) $.
Let $P_{\lambda}$ be the projection onto the vector spanned by $\frac{\lambda}{| \lambda |}$. Then $|\lambda|^2 P_{\lambda} = \lambda \otimes \lambda$ and the Hessian matrix of $\psi$ is given by
\begin{align*}
\hess \psi ( \lambda ) 
=& ( |\lambda|^2 + |\rho|^2)^{-\frac{1}{2}} I_{\ell} -  ( |\lambda|^2 + |\rho|^2)^{-\frac{3}{2}} \lambda \otimes \lambda \\
=&  ( |\lambda|^2 + |\rho|^2)^{-\frac{3}{2}}  \big\lbrace |\rho|^2P_{\lambda} + ( |\lambda|^2 + |\rho|^2) (I_{\ell} - P_{\lambda}) \big\rbrace \\
=&
 ( |\lambda|^2 + |\rho|^2)^{-\frac{3}{2}} 
\left(
\begin{array}{c|c}
| \rho|^2  & 0 \\
 \hline
{ } & { } \\
0 &  ( |\lambda|^2 + |\rho|^2) I_{\ell-1} \\
{ } & { } \\
\end{array}
\right),
\end{align*}
which is a positive definite symmetric matrix. Hence $ \lambda_{0}$ is a nondegenerate critical point. Since $\nabla \psi ( \lambda_{0} ) = 0$, we have
\begin{align*}
\psi ( \lambda ) - \psi ( \lambda_{0} ) =  ( \lambda - \lambda_{0})^{T}
\Big\lbrace
\underbrace{\int_{0}^{1} ds \ (1-s)   \hess \psi \big( \lambda_{0} + s ( \lambda - \lambda_{0} ) \big) }_{= M(\lambda)}
\Big\rbrace
(\lambda -  \lambda_{0} ),
\end{align*}
where  $M(\lambda)$ belongs, for every $\lambda \in B(\lambda_{0} , \eta)$, to a compact subset of the set of positive definite symmetric matrices.
We introduce a new variable $\mu = M(\lambda)^{\frac{1}{2}} (\lambda -  \lambda_{0})$, then $|\mu|^{2} = \psi( \lambda) - \psi (\lambda_{0})$ and $\mu = 0$ if and only if $\lambda = \lambda_{0}$. Notice that for every $k \in \mathbb{N}$, there exists $C_{k} > 0$ such that
\begin{align}\label{estimate of M}
| \nabla^{k}  M(\lambda)^{\frac{1}{2}} | \le C_k
\qquad \forall \lambda \in B(\lambda_{0}, \eta).
\end{align}
Denote by $J(\lambda)$ the Jacobian matrix such that $d \mu = \det \left[  J(\lambda) \right] d \lambda$, then we can choose $\eta >0 $ small enough such that
\begin{align}\label{estimate of J}
\det \left[  J(\lambda)  \right] > \frac{1}{2} \det \left[  M(\lambda)^{\frac{1}{2}}  \right]
\qquad \forall \lambda \in B(\lambda_{0} , \eta).
\end{align}
{ }\\

Now let us divide the study of $I_{0}(t,x)$ into two parts, corresponding to the decomposition
\begin{align*}
I_{0}(t,x) = 
\underbrace{\int_{\mathfrak{a}} \ d \lambda \ \chi_{\eta}^{-}(\lambda) a(\lambda) \ e^{it \psi (\lambda)}}_{ = I_{0}^{-}(t,x)} +
\underbrace{\int_{\mathfrak{a}} \ d \lambda \ \chi_{\eta}^{+}(\lambda) a(\lambda) \ e^{it \psi (\lambda)}}_{ = I_{0}^{+}(t,x)},
\end{align*}
associated with the smooth cut-off functions
\begin{align*}
\chi_{\eta}^{-}(\lambda) = \chi_{0} \Big( \frac{|\lambda - \lambda_{0}|}{2\eta} \Big)
\qquad \textnormal{and} \qquad 
\chi_{\eta}^{+}(\lambda) = \chi_{\infty} \Big( \frac{|\lambda - \lambda_{0}|}{2\eta} \Big).
\end{align*}
{ }\\

\noindent {\bf Estimate of $ I_{0}^{-}(t,x)$.}
Notice that $\supp \chi_{\eta}^{-} \subset B(\lambda_{0} , \eta)$. There exist $0<\widetilde{\eta}_{1}\le\widetilde{\eta}_{2}$ such that $\mu \in B(0,\widetilde{\eta}_{1})$ implies $\lambda \in B(\lambda_{0} , \eta)$, and $\lambda \in B(\lambda_{0} , \eta)$ implies  $\mu \in B(0,\widetilde{\eta}_{2})$. By substituting $\psi( \lambda)  = |\mu|^{2} + \psi ( \lambda_{0})$, we get
\begin{align*}
I_{0}^{-}(t,x) = e^{it \psi (\widetilde{\lambda)}} \int_{\mathfrak{a}} \ d \mu \ \widetilde{a} (\lambda(\mu)) e^{it |\mu|^{2}}
\end{align*}
where the amplitude
\begin{align}\label{tilde a}
\widetilde{a} (\lambda(\mu)) = 
\chi_{\eta}^{-}(\lambda(\mu)) \ \chi_{0}^{\rho} (\lambda(\mu)) \pi (\lambda(\mu))^2 ( | \lambda(\mu)|^2 + \widetilde{\rho}^2)^{ - \frac{\sigma}{2}} \det \left[ J (\lambda(\mu)) \right]^{-1}
\end{align}
is smooth and compactly supported in $B(0,\widetilde{\eta}_{2})$. We deduce, from \eqref{estimate of M} and \eqref{estimate of J} that $\widetilde{a} (\lambda(\mu))$ is bounded, together with all its derivatives. Let $\widetilde{\chi} \in \mathcal{C}_{c}^{\infty}(\mathfrak{a})$ be a bump function such that $\widetilde{\chi}=1$ on $B(0,\widetilde{\eta}_{2} )$. Let $M\ge\frac{d}{2}$ be an integer. Then
\begin{align*}
I_{0}^{-}(t,x) =  
e^{ it \psi (\lambda_{0})} 
\int_{\mathfrak{a}} \ d \mu \ e^{it |\mu|^{2}} e^{- |\mu|^{2}} \big[ e^{|\mu|^{2}}  \widetilde{a} (\lambda(\mu))\big] \widetilde{\chi}(\mu),
\end{align*}
where the coefficients of the Taylor expansion of ${e^{|\mu|^{2}} \tilde{a}(\lambda(\mu))}$ at the origin: 
\begin{align*}
{e^{|\mu|^{2}} \tilde{a}(\lambda(\mu))} =
\sum_{|k|\le2M} c_k \mu^{k} + R_{2M}(\mu),
\end{align*}
satisfy 
\begin{align}\label{c_0}
\textstyle
|c_k| \lesssim |\pi(\lambda_{0})|^{2}\lesssim\big( \frac{|x|}{t} \big)^{d-\ell},
\end{align}
and the remainder 
\begin{align}\label{R_2M}
    |\nabla^{k} R_{2M}(\mu)| \lesssim |\mu|^{2M+1-k}
    \qquad \forall 0\le k\le2M+1,
\end{align}
according to \eqref{critical point}, \eqref{tilde a} and the fact that $\frac{|x|}{t}\le\frac{1}{2}$. By substituting this expansion in the above integral, $I_{0}^{-}(t,x)$ is equal to the sum of following three terms:
\begin{align*}
I_1 = \sum_{|k|\le2M} c_k \int_{\mathfrak{a}} d\mu \ e^{i t |\mu|^{2}}  e^{- |\mu|^{2}} \mu^{k}, \\
I_2 = \int_{\mathfrak{a}} d\mu \ e^{i t |\mu|^{2}}  R_{2M}(\mu) e^{- |\mu|^{2}} \tilde{\chi} (\mu),
\end{align*}
and
\begin{align*}
I_3  = \int_{\mathfrak{a}} d\mu \ e^{i t |\mu|^{2}} \Big( \sum_{|k|\le2M} c_k \mu^{k} \Big) e^{- |\mu|^{2}}  \Big(  \tilde{\chi} (\mu) - 1 \Big).
\end{align*}
{ }\\

In order to estimate $I_1$, we write the integral as a product
\begin{align*}
I_1 =& \sum_{|k|\le2M} c_k \prod_{j=1}^{\ell} 
\int_{- \infty}^{+ \infty} d \mu_{j} \ e^{it \mu_{j}^{2}} e^{- \mu_{j}^{2}} \mu_{j}^{k_{j}},
\end{align*}
where
\begin{align*}
    \int_{-\infty}^{+\infty} d\mu_{j} \ e^{-(1-it)\mu_{j}^{2}} \mu_{j}^{k_{j}} =0
\end{align*}
if $k_{j}$ is odd, while
\begin{align*}
    \int_{-\infty}^{+\infty} d\mu_{j} \ e^{-(1-it)\mu_{j}^{2}} \mu_{j}^{k_{j}} 
    = 2 (1-it)^{-\frac{k_{j}+1}{2}} \int_{0}^{+\infty} dz_{j} e^{-z_{j}^2} z_{j}^{k_{j}}
\end{align*}
by a change of contour if $k_{j}$ is even. From \eqref{c_0}, we obtain
\begin{align*}
\textstyle
    |I_1| \lesssim \big(\frac{|x|}{t}\big)^{d-\ell} t^{-\frac{\ell}{2}}
    = t^{-\frac{d}{2}} \big(\frac{|x|}{\sqrt{t}}\big)^{d-\ell}
    \lesssim t^{-\frac{d}{2}} |x|^{-\frac{d-\ell}{2}},
\end{align*}
since $\frac{|x|}{t}<\frac{1}{2}$. Next, we perform $M$ integrations by parts based on
\begin{align}\label{IBP I0}
\textstyle
    e^{i t |\mu|^{2}}
    = - \frac{i}{2t} \sum_{j=1}^{\ell} \frac{\mu_{j}}{|\mu|^2}
     \frac{\partial}{\partial \mu_{j}} e^{i t |\mu|^{2}}
\end{align}
and obtain
\begin{align*}
    I_{2} = O(t^{-M}),
\end{align*}
according to \eqref{R_2M}. Finally, as $ \mu \mapsto  \big( \sum_{|k|\le2M} c_k \mu^{k} \big) e^{- |\mu|^{2}}  \big(  \tilde{\chi} (\mu) - 1 \big)$ is exponentially decreasing and vanishes near the origin, we perform $N\ge\frac{d}{2}$ integrations by parts based on \eqref{IBP I0} again and obtain
\begin{align*}
    I_{3} = O(t^{-N}).
\end{align*}
By summing up the estimates of $I_1$, $I_2$ and $I_3$, we deduce that
\begin{align}\label{I0-}
| I_{0}^{-}(t,x) | \lesssim t^{-\frac{d}{2}} (1+|x|)^{\frac{d-\ell}{2}}.
\end{align}
{ }\\

\noindent {\bf Estimate of $ I_{0}^{+}(t,x)$.}
Since the phase $\psi$ has a unique critical point $\lambda_{0}$, then for all $\lambda \in \supp \chi_{\eta}^{+}$, $\nabla \psi (\lambda) \neq 0$. In order to get a large time decay, we estimate the oscillatory integral  $ I_{0}^{+}(t,x)$ by using several integrations by parts based on 
\begin{align}\label{int by parts}
e^{it \psi( \lambda)} = \frac{1}{it} 
\big( \widetilde{\psi}(\lambda) \big)^{-1}
\sum_{j=1}^{\ell} \lambda_{j} \frac{\partial}{\partial \lambda_{j}} e^{it \psi( \lambda)},
\end{align}
where
\begin{align*}
\textstyle
\widetilde{\psi}(\lambda) =
| \lambda |^2 ( | \lambda |^2 + | \rho |^2 )^{- \frac{1}{2}} + \big\langle {\frac{A}{t} , \lambda} \big\rangle
\end{align*}
is a nonzero symbol of order zero for all $\lambda \in \supp \chi_{\eta}^{+}$.
After performing $N$ such integrations by parts, $ I_{0}^{+}(t,x)$ becomes
\begin{align*}
I_{0}^{+}(t,x) = const. & (it)^{-N} \\ 
& \underbrace{\int_{\mathfrak{a}} \ d \lambda \ e^{it \psi (\lambda)} \Big\lbrace - \sum_{j=1}^{\ell} \frac{\partial}{\partial \lambda_{j}} \circ \frac{\lambda_{j}}{\widetilde{\psi}(\lambda) } \Big\rbrace^{N} \Big\lbrace \chi_{\eta}^{+}(\lambda) a(\lambda) \Big\rbrace}_{< + \infty},
\end{align*}
for every $N \in \mathbb{N}$. The last integral is finite since the amplitude $a$ is supported in $B(0,2|\rho|)$ thanks to the cut-off $\chi_{0}^{\rho}$. Hence
\begin{align}\label{I0+}
| I_{0}^{+}(t,x) | \lesssim t^{- N},
\end{align}
for every integer $N \ge 0$. According to \eqref{I0-} and \eqref{I0+}, we conclude that
\begin{align*}
|\omega_{t}^{\sigma,0} (x) | \lesssim | I_{0} (t,x) |
\int_{K} dk \ e^{- \langle \rho, A (xk) \rangle} \lesssim  t^{-\frac{d}{2}} |x|^{\frac{d-\ell}{2}} \varphi_{0} (x ).
\end{align*}
\end{proof}

\begin{remark}
In the proof of \cref{estimate omega 0}, we have replaced the  spherical function $\varphi_{\lambda}(x)$ by its general integral expression \eqref{spherical function noncompact}. Moreover, this proof also works without using the explicit expression \eqref{c function complex} of the density $| \mathbf{c} (\lambda) |^{-2}$. Thus we can derive similar estimates for any symmetric space $G/K$ by the same approach.
\end{remark}

\begin{theorem}\label{estimate omega infinity}
Let $\sigma \in \mathbb{C}$ with $\re \sigma = \frac{d+1}{2}$. The following pointwise estimates hold for the kernel $\widetilde{\omega}_{\infty}^{\sigma}(t, x)$, for all $t \in \mathbb{R}^{*}$ and $ x \in \mathfrak{a}$:
\begin{align*}
| \widetilde{\omega}_{t}^{\sigma, \infty}(x) | \lesssim
\begin{cases}
|t|^{-\frac{d-1}{2}} \varphi_{0}(x) \qquad & \textnormal{if}\quad |t| < 1, \\
|t|^{-N} (1+ |x|)^{N} \varphi_{0}(x) \qquad & \textnormal{if}\quad |t| \ge 1, 
\end{cases}
\end{align*}
for every $N \in \mathbb{N}$. 
\end{theorem}

\begin{remark}
Similar estimates hold in the easier case $\re \sigma > \frac{d+1}{2}$.  By symmetry we assume again that $t >0$. We divide the proof into two parts, depending whether $\frac{|x|}{t} \ge \frac{1}{2}$ or not.
\end{remark}

\begin{proof}[Proof for $\frac{|x|}{t} \ge \frac{1}{2}$]
Substituting the integral representation \eqref{spherical function complex} of $\varphi_{\lambda}$ allows us to write the kernel as
\begin{align*}
\widetilde{\omega}_{t}^{\sigma, \infty}(x)
=  const. \ & \frac{e^{\sigma^2}}{\Gamma ((d+1)/2- \sigma)} \ \varphi_{0} (x) \\
& \underbrace{\int_{\mathfrak{p}} \ d \lambda \  \chi_{\infty}^{\rho} (\lambda)  ( | \lambda |^2 + \widetilde{\rho}^2)^{ - \frac{\sigma}{2}} e^{it \sqrt{| \lambda|^2 + |\rho|^2}} e^{- i \left\langle {x, \lambda} \right\rangle}}_{ = I_{\infty} (t,x) }.
\end{align*}
We split up $I_{\infty} (t,x) $ into two parts
\begin{align*}
I_{\infty} (t,x)
=& \overbrace{\int_{\mathfrak{p}} \ d \lambda \  \chi_{0}^{\rho} (t \lambda)  \chi_{\infty}^{\rho} (\lambda)  ( | \lambda |^2 + \widetilde{\rho}^2)^{ - \frac{\sigma}{2}} e^{it \sqrt{| \lambda|^2 + |\rho|^2}} e^{- i \left\langle {x, \lambda} \right\rangle}}^{I_{\infty}^{-} (t,x)} \\
&+ \underbrace{\int_{\mathfrak{p}} \ d \lambda \  \chi_{\infty}^{\rho} (t \lambda) \chi_{\infty}^{\rho} (\lambda)  ( | \lambda |^2 + \widetilde{\rho}^2)^{ - \frac{\sigma}{2}} e^{it \sqrt{| \lambda|^2 + |\rho|^2}} e^{- i \left\langle {x, \lambda} \right\rangle}}_{ I_{\infty}^{+} (t,x) },
\end{align*}
by using the smooth radial cut-off functions $\chi_{0}^{\rho} (t \lambda)$ and $\chi_{\infty}^{\rho} (t \lambda)$. The first integral $I_{\infty}^{-} (t,x)$ vanishes if $t\ge2$ and is easily estimated as follows, if $0<t<2$ :
\begin{align*}
| I_{\infty}^{-} (t,x) | 
\lesssim& \ \int_{\mathfrak{p}} \ d \lambda \  \chi_{0}^{\rho} (t \lambda)  \chi_{\infty}^{\rho} (\lambda)  | \lambda |^{- \re \sigma} \\
\le& \ t^{- \frac{d-1}{2}}  \int_{\mathfrak{p}} \ d \lambda \  \chi_{0}^{\rho} ( \lambda)   |\lambda|^{- \frac{d+1}{2}} \lesssim t^{- \frac{d-1}{2}} 
\end{align*}
for all $ \sigma \in \mathbb{C}$ with $\re \sigma =  \frac{d+1}{2}$. In order to study the second integral $I_{\infty}^{+} (t,x)$, we introduce polar coordinates and rewrite
\begin{align}\label{I inf +}
I_{\infty}^{+} (t,x) =
\int_{0}^{+\infty} dr \ \chi_{\infty}^{\rho} (t r) \chi_{\infty}^{\rho} (r) ( r^2 + \widetilde{\rho}^2)^{ - \frac{\sigma}{2}} e^{it \sqrt{r^2 + |\rho|^2}} \int_{|\lambda| =r} d \sigma(\lambda) \ e^{- i \left\langle {x, \lambda} \right\rangle}.
\end{align}
Here the inner integral is a modified Bessel function. Specifically
\begin{align}\label{integral on the shell}
\int_{|\lambda| =r} d \sigma(\lambda) \ e^{- i \left\langle {x, \lambda} \right\rangle} = r^{\frac{d}{2}} |x|^{\frac{2-d}{2}} J_{\frac{d-2}{2}} (r|x|),
\end{align}
where $J_{(d-2)/2}$ denotes the classical Bessel function of the first kind. Since $I_{\infty}^{+} (t,x)$ vanishes if $tr<|\rho|$, we may restrict ourselves to $tr\ge|\rho|$, hence $r|x|\ge\frac{rt}{2}\ge\frac{|\rho|}{2}$. Then from \eqref{Bessel function asymp}, there exist constants $C_{d}^{+}$ and $C_{d}^{-}$ such that
\begin{align}\label{Bessel function}
J_{\frac{d-2}{2}} (r|x|) = C_{d}^{+} (r|x|)^{-1/2} e^{ir|x|} + C_{d}^{-} (r|x|)^{-1/2}e^{-ir|x|} + R(r|x|),
\end{align}
where the remainder $R(r|x|)$ is $O \big( (r|x|)^{-\frac{3}{2}}\big)$. Substituting \eqref{integral on the shell} and \eqref{Bessel function} in \eqref{I inf +} implies
\begin{align*}
I_{\infty}^{+} (t,x) 
=& I_{\infty}^{+,1} (t,x) + I_{\infty}^{+,2} (t,x) + I_{\infty}^{+,3} (t,x) \\
=& C_{d}^{+}  |x|^{-\frac{d-1}{2}} \int_{0}^{+\infty} dr \ \chi_{\infty}^{\rho} (t r) \chi_{\infty}^{\rho}(r) ( r^2 + \widetilde{\rho}^2)^{ - \frac{\sigma}{2}} r^{\frac{d-1}{2}} e^{it \sqrt{r^2 + |\rho|^2}} e^{ i|x| r} \\
& + C_{d}^{-} |x|^{-\frac{d-1}{2}} \int_{0}^{+\infty} dr \ \chi_{\infty}^{\rho} (t r) \chi_{\infty}^{\rho}(r) ( r^2 + \widetilde{\rho}^2)^{ - \frac{\sigma}{2}} r^{\frac{d-1}{2}} e^{it \sqrt{r^2 + |\rho|^2}} e^{- i|x| r} \\
& + |x|^{-\frac{d-2}{2}} \int_{0}^{+\infty} dr \ \chi_{\infty}^{\rho} (t r) \chi_{\infty}^{\rho}(r) ( r^2 + \widetilde{\rho}^2)^{ - \frac{\sigma}{2}} r^{\frac{d}{2}} e^{it \sqrt{r^2 + |\rho|^2}}  R(r|x|).
\end{align*}
On the one hand, the last integral is estimated as follows:
\begin{align*}
| I_{\infty}^{+,3} (t,x)  | \lesssim
|x|^{-\frac{d+1}{2}} \int_{0}^{+\infty} dr \ \chi_{\infty}^{\rho} (t r) \chi_{\infty}^{\rho}(r) \ r^{-2},
\end{align*}
for all $ \sigma \in \mathbb{C}$ with $\re \sigma =  \frac{d+1}{2}$. Hence
\begin{align}\label{I+3 t>1}
| I_{\infty}^{+,3} (t,x)  | \lesssim
|x|^{-\frac{d+1}{2}} \int_{|\rho|}^{+\infty} dr \ r^{-2} \lesssim |x|^{-\frac{d+1}{2}}
\end{align}
for large time and
\begin{align}\label{I+3 t<1}
| I_{\infty}^{+,3} (t,x)  | \lesssim
|x|^{-\frac{d+1}{2}} \int_{|\rho| t^{-1}}^{+\infty} dr \ r^{-2}
\lesssim t |x|^{-\frac{d+1}{2}} \lesssim |x|^{-\frac{d-1}{2}}
\end{align}
for small time.\\

On the other hand, in order to estimate $I_{\infty}^{+,1} (t,x)$ and $I_{\infty}^{+,2} (t,x)$, it is sufficient to study the integral
\begin{align*}
\widetilde{I}_{\infty}^{+} (t,x) = \int_{0}^{+\infty} dr \ \chi_{\infty}^{\rho}(r) \chi_{\infty}^{\rho} (t r)  ( r^2 + \widetilde{\rho}^2)^{ - \frac{\sigma}{2}} r^{\frac{d-1}{2}} e^{it \sqrt{r^2 + |\rho|^2}} e^{\pm i|x| r}.
\end{align*}
Notice that
\begin{align}\label{prod of cutoff}
 \chi_{\infty}^{\rho} (r) \chi_{\infty}^{\rho}(tr) =
 \begin{cases}
 \chi_{\infty}^{\rho} (r) \qquad & \textnormal{if}\quad t \ge 2, \\
\chi_{\infty}^{\rho}(tr) \qquad & \textnormal{if}\quad t \le \frac{1}{2}.
 \end{cases}
\end{align}
We split up $\widetilde{I}_{\infty}^{+} (t,x)$ into three parts, corresponding to the decomposition
\begin{align*}
\widetilde{I}_{\infty}^{+} (t,x) = \widetilde{I}_{\infty}^{+,1} (t,x)  + \widetilde{I}_{\infty}^{+,2} (t,x)  + \widetilde{I}_{\infty}^{+,3} (t,x) 
\end{align*}
where
\begin{align*}
\widetilde{I}_{\infty}^{+,1} (t,x) =  \int_{0}^{+\infty}dr \ \chi_{\infty}^{\rho}  (r) \chi_{\infty}^{\rho}(tr) ( r^2 + \widetilde{\rho}^2)^{ - \frac{\sigma}{2}} r^{\frac{d-1}{2}} \big\lbrace e^{it \sqrt{r^2 + |\rho|^2}} - e^{ it r} \big\rbrace e^{\pm i|x| r},
\end{align*}
\begin{align*}
\widetilde{I}_{\infty}^{+,2} (t,x) =  \int_{0}^{+\infty} dr \ \chi_{\infty}^{\rho} (r) \chi_{\infty}^{\rho}(tr) r^{\frac{d-1}{2}} \big\lbrace  ( r^2 + \widetilde{\rho}^2)^{ - \frac{\sigma}{2}} -r^{- \sigma} \big\rbrace e^{i (t \pm |x|) r},
\end{align*}
and
\begin{align*}
\widetilde{I}_{\infty}^{+,3} (t,x) = \int_{0}^{+\infty} dr \ \chi_{\infty}^{\rho} (r) \chi_{\infty}^{\rho}(tr) \ r^{-1-i \im \sigma} e^{i (t \pm |x|) r},
\end{align*}
for all $\sigma \in \mathbb{C}$ with $\re \sigma = \frac{d+1}{2}$.\\

\noindent {\bf Estimates of $\widetilde{I}_{\infty}^{+,1}$ and $\widetilde{I}_{\infty}^{+,2}$.} The first two parts are easily estimated. By using
\begin{align*}
\textstyle
e^{it \sqrt{r^2 + |\rho|^2}} - e^{ it r} = O ( \frac{t}{r} )
\qquad \textnormal{and} \qquad
 ( r^2 + \widetilde{\rho}^2)^{ - \frac{\sigma}{2}} -r^{- \sigma} = O (r^{- \re \sigma - 2})
\end{align*}
for all $r \in \supp \chi_{\infty}^{\rho}$, we obtain
\begin{align}\label{estimate of tilde I+1 and I+2}
\widetilde{I}_{\infty}^{+,1}(t,x) = O(t)
\qquad \textnormal{and} \qquad
\widetilde{I}_{\infty}^{+,2}(t,x) = O(1)
\end{align}
for all $\sigma \in \mathbb{C}$ with $\re \sigma = \frac{d+1}{2}$.\\

\noindent {\bf Estimate of $\widetilde{I}_{\infty}^{+,3}$.} 
 We claim that 
\begin{align}\label{estimate of tilde I+3}
|\widetilde{I}_{\infty}^{+,3} (t,x) | \lesssim 
\underbrace{|\im \sigma | + \frac{1}{|\im \sigma |}}_{C_{\sigma}}.
\end{align}
Denoting by $\xi = t \pm |x|$, we divide our estimate into two parts, depending whether $| \xi | \ge \frac{1}{2 | \rho |}$ or not. \\

\begin{itemize}[leftmargin=*]
\item If $| \xi | \ge \frac{1}{2 | \rho |}$, we perform an integration by parts based on $e^{i \xi r} = \frac{1}{i \xi} \frac{\partial}{\partial r} e^{i \xi r}$. Then
\begin{align}\label{case 2rhoxi>1}
\widetilde{I}_{\infty}^{+,3} (t,x) = \frac{i}{\xi}  \ \int_{0}^{+ \infty} dr \ e^{i \xi r} \ \frac{\partial}{\partial r} \big\lbrace \chi_{\infty}^{\rho} (r) \chi_{\infty}^{\rho}(tr) \ r^{-1-i \im \sigma}  \big\rbrace.
\end{align}
If the derivative hits either $\chi_{\infty}^{\rho} (r)$ or $r^{-1-i \im \sigma}$, then the corresponding contributions to \eqref{case 2rhoxi>1} are bounded by
\begin{align*}
    C(1+|\im\sigma|).
\end{align*}
If the derivative hits $\chi_{\infty}^{\rho}(tr)$, which matters only in the case $t<2$, according to \eqref{prod of cutoff}, then the corresponding contribution to \eqref{case 2rhoxi>1} is bounded by
\begin{align*}
Ct \int_{| \rho| t^{-1}}^{2| \rho| t^{-1}} dr \ r^{-1} < 2(\ln2)C.
\end{align*}
Hence $\widetilde{I}_{\infty}^{+,3} (t,x) = O (1 + | \im \sigma| )$ for all $t >0$ and $x \in \mathfrak{a}$ satisfying $\frac{|x|}{t} \ge \frac{1}{2}$ and $\big| t \pm |x| \big| \ge \frac{1}{2 | \rho |}$.\\

\item If $|\xi| < \frac{1}{2 | \rho |}$, we split up
\begin{align*}
\widetilde{I}_{\infty}^{+,3} (t,x) = \big( \int_{0}^{\frac{1}{| \xi |}} + \int_{\frac{1}{| \xi |}}^{+ \infty} \big) \ dr \ \chi_{\infty}^{\rho} (r) \chi_{\infty}^{\rho}(tr) \ r^{-1-i \im \sigma} e^{i \xi r},
\end{align*}
We estimate both terms by performing different integration by parts. On the one hand,
\begin{align*}
& \int_{0}^{\frac{1}{| \xi |}} dr \ r^{-1-i \im \sigma}  \chi_{\infty}^{\rho} (r) \chi_{\infty}^{\rho}(tr) \ e^{i \xi r} \\
=& \frac{i}{\im \sigma} \ \underbrace{r^{- i \im \sigma} \chi_{\infty}^{\rho} (r) \chi_{\infty}^{\rho}(tr) e^{i \xi r} \big|_{r=0}^{r= \frac{1}{| \xi |}}}_{ = \ O(1) } \\
& \qquad  \qquad \qquad \qquad -  \frac{i}{\im \sigma} \underbrace{\int_{0}^{\frac{1}{| \xi |}} dr \ r^{-i \im \sigma}  \frac{\partial}{\partial r} \big\lbrace \chi_{\infty}^{\rho} (r) \chi_{\infty}^{\rho}(tr) e^{i \xi r} \big\rbrace}_{= \ O(1) \quad (*)} \\
=& O\big( \frac{1}{| \im \sigma |} \big).
\end{align*}
On the other hand,
\begin{align*}
& \int_{\frac{1}{|\xi|}}^{+ \infty} dr \ e^{i \xi r} \chi_{\infty}^{\rho} (r) \chi_{\infty}^{\rho}(tr) \ r^{-1-i \im \sigma}  \\
=& \underbrace{\frac{1}{i\xi} e^{i \xi r} \chi_{\infty}^{\rho} (r) \chi_{\infty}^{\rho}(tr) \ r^{- 1 - i \im \sigma} \big|_{r= \frac{1}{| \xi |}}^{r = + \infty}}_{ = \ O(1) } \\
& \qquad  \qquad \qquad \qquad +  \underbrace{ \frac{i}{\xi} \int_{\frac{1}{| \xi |}}^{+ \infty} dr \ e^{i \xi r}  \frac{\partial}{\partial r} \big\lbrace \chi_{\infty}^{\rho} (r) \chi_{\infty}^{\rho}(tr) \ r^{- 1 - i \im \sigma} \big\rbrace}_{= \ O(1 + | \im \sigma |) \quad (**)} \\
=& O(1 + | \im \sigma |).
\end{align*}
Let us explain the estimates $(*)$ and $(**)$. In both cases, the derivative $\frac{\partial}{\partial r}$ produces $3$ or $2$ terms, depending whether $t \in (\frac{1}{2},2)$ or not. The contributions of $\frac{\partial}{\partial r} e^{i \xi r} = i \xi e^{i \xi r}$ to $(*)$ and of $\frac{\partial}{\partial r} r^{-1 - i \im \sigma} = - (1 + i \im \sigma) r^{-2 - i \im \sigma}$ to $(**)$ are easily dealt with. The contribution of $\frac{\partial}{\partial r} \chi_{\infty}^{\rho}(r)$, whose support is contained in the interval $\left[ {|\rho|, 2 |\rho|} \right]$, to $(*)$ is obviously bounded, while its contribution to $(**)$ vanishes, as $\frac{1}{| \xi |} > 2 | \rho |$. 

Consider finally the contributions of $\frac{\partial}{\partial r}  \chi_{\infty}^{\rho}(tr)  = t (  \chi_{\infty}^{\rho})' (tr)$ when $t$ is small. Regarding $(*)$, it is easily estimated by $t \int_{| \rho | t^{-1}}^{2 | \rho | t^{-1}} dr = | \rho |$. Regarding $(**)$, notice that it vanishes unless $|\xi| \ge \frac{t}{2| \rho|}$. In this case, it is estimated by 
\begin{align*}
\frac{t}{|\xi| } \int_{| \rho | t^{-1}}^{2 | \rho | t^{-1}} dr r^{-1} \le 2|\rho| \ln 2.
\end{align*}
This concludes the estimate of \eqref{estimate of tilde I+3}.\\
\end{itemize}

We deduce from \eqref{estimate of tilde I+1 and I+2} and \eqref{estimate of tilde I+3} that
\begin{align*}
\widetilde{I}_{\infty}^{+} (t,x) = O (t +C_{\sigma}).
\end{align*}
By combining this estimate with \eqref{I+3 t>1} and  \eqref{I+3 t<1}, we get next
\begin{align*}
    |I_{\infty}^{+}| \lesssim C_{\sigma}|x|^{-\frac{d-1}{2}}
    \lesssim C_{\sigma}t^{-\frac{d-1}{2}},
\end{align*}
if $t$ is small and
\begin{align*}
    |I_{\infty}^{+}| \lesssim C_{\sigma} (t|x|^{-\frac{d-1}{2}} +|x|^{-\frac{d+1}{2}})
    \lesssim  C_{\sigma} t^{-N} (1+|x|)^{N-\frac{d-3}{2}},
\end{align*}
if $t$ is large, for any $N\in\mathbb{N}$. Together with the estimate of $I_{\infty}^{-}$, this allows us to conclude the proof of \cref{estimate omega infinity} in the case $\frac{|x|}{t}\ge\frac{1}{2}$.
\end{proof}

\begin{remark}
This proof relies on the expression \eqref{spherical function complex} of spherical functions, which is available if $G$ is complex but not in general.\\
\end{remark}

\begin{proof}[Proof for $\frac{|x|}{t} \le \frac{1}{2}$]
Substituting the expressions \eqref{spherical function noncompact} of $\varphi_{\lambda}$ and \eqref{c function complex} of  $ | \mathbf{c} ( \lambda)|^{-2} $ implies
\begin{align*}
\widetilde{\omega}_{t}^{\sigma, \infty}(x) = const. \  \frac{e^{\sigma^2}}{\Gamma ((d+1)/2- \sigma)} & \int_{K} dk \ e^{ \langle \rho, A (xk) \rangle} \\
& \underbrace{\int_{\mathfrak{a}} \ d \lambda \ \chi_{\infty}^{\rho} (\lambda) \ \pi^2(\lambda) \ ( | \lambda |^2 + \widetilde{\rho}^2)^{ - \frac{\sigma}{2}} e^{it \psi (\lambda)}}_{ = \widetilde{II}_{\infty} (t,x)},
\end{align*}
where the phase $\psi$ is defined as $\eqref{Sect 3. phase}$.
Since $\frac{| A (xk)| }{t} \le \frac{|x|}{t} \le \frac{1}{2}$, we know from \eqref{critical point} that the critical point $\lambda_{0}$ of $\psi$ satisfies $ | \lambda_{0}| < \frac{|\rho|}{\sqrt{3}}$. Hence $\psi$ has no critical point in the support of $\chi_{\infty}^{\rho}$, and $\widetilde{II}_{\infty} (t,x)$ makes sense after several integrations by parts based on \eqref{int by parts}.\\

Consider first the large time case where $t \ge 1$. After performing $N$ integrations by parts based on \eqref{int by parts}, $\widetilde{II}_{\infty} (t,x)$ becomes
\begin{align*}
\widetilde{II}_{\infty} (t,x) = &
\ const. (it)^{-N} \\
& \int_{\mathfrak{a}} d\lambda \ e^{it \psi (\lambda)} \Big\lbrace -  \sum_{j=1}^{\ell} \frac{\partial}{\partial \lambda_{j}} \circ \frac{\lambda_{j}}{\widetilde{\psi}(\lambda)} \Big\rbrace^{N} \Big\lbrace \chi_{\infty}^{\rho} (\lambda) \ \pi^2(\lambda) \ ( | \lambda |^2 + \widetilde{\rho}^2)^{ - \frac{\sigma}{2}} \Big\rbrace.
\end{align*}
If some derivatives hit $\chi_{\infty}^{\rho} (\lambda)$, then the range of integration reduces to a spherical shell where $|\lambda| \asymp | \rho |$ and we obtain the bound $O(t^{-N})$. Otherwise, since (in the support of $\chi_{\infty}^{\rho}$) $\lambda_{j} \big( \widetilde{\psi}(\lambda) \big)^{-1}$, $\pi^2(\lambda)$ and $( | \lambda |^2 + \widetilde{\rho}^2)^{ - \frac{\sigma}{2}}$ are inhomogeneous symbols of order $0$, $2|\Sigma^{+}$ and $-\re\sigma=-\frac{d+1}{2}$ respectively, we have
\begin{align*}
| \widetilde{II}_{\infty} (t,x) | \lesssim t^{-N} 
\int_{\mathfrak{a}} d \lambda \  \chi_{\infty} (\lambda) |\lambda|^{2 | \Sigma^{+}| - \re \sigma - N} \lesssim  t^{-N} 
\end{align*}
provided that $N > \re \sigma =\frac{d+1}{2}$. Hence, for all $t \ge 1$ such that $ \frac{|x|}{t} \le \frac{1}{2}$,
\begin{align*}
| \widetilde{\omega}_{t}^{\sigma, \infty}(x) | \lesssim \
t^{-N}  \int_{K} dk \ e^{ \langle \rho, A (xk) \rangle} = 
t^{-N} \varphi_{0}(x), 
\end{align*}
for all $\sigma \in \mathbb{C}$ with $\re \sigma =\frac{d+1}{2}$, and for every integer $N > \frac{d+1}{2}$.\\

Notice that the same argument implies similar estimates in small time with $t^{- \frac{d+1}{2}}$, where the negative power is too large. In order to reduce it to $t^{- \frac{d-1}{2}}$, we split up $\widetilde{II}_{\infty} (t,x)$ into two parts
\begin{align*}
\widetilde{II}_{\infty} (t,x) =& \
\widetilde{II}_{\infty}^{-} (t,x) + \widetilde{II}_{\infty}^{+} (t,x) \\
=& 
\int_{\mathfrak{a}} \ d \lambda \ \chi_{0}^{\rho}  ( t \lambda) \chi_{\infty}^{\rho}  (\lambda) \ \pi^2(\lambda) \ ( | \lambda |^2 + \widetilde{\rho}^2)^{ - \frac{\sigma}{2}} e^{it \psi (\lambda)} \\
& +  \int_{\mathfrak{a}} \ d \lambda \ \chi_{\infty}^{\rho}  ( t \lambda) \chi_{\infty}^{\rho}  (\lambda) \ \pi^2(\lambda) \ ( | \lambda |^2 + \widetilde{\rho}^2)^{ - \frac{\sigma}{2}} e^{it \psi (\lambda)},
\end{align*}
where $\widetilde{II}_{\infty}^{-} (t,x)$ is easily estimated:
\begin{align}\label{I tilde inf -}
| \widetilde{II}_{\infty}^{-} (t,x) | \lesssim 
\int_{|\rho| \le | \lambda | \le 2 | \rho | t^{-1}} \ d \lambda \ | \lambda |^{2 | \Sigma^{+} | -   \re \sigma} \lesssim t^{- 2 | \Sigma^{+} | +  \re \sigma - \ell} = t^{- \frac{d-1}{2}},
\end{align}
for all $0<t<1$ and for all $\sigma \in \mathbb{C}$ with $\re \sigma = \frac{d+1}{2}$. After $M$ integrations by parts based on \eqref{int by parts}, $\widetilde{II}_{\infty}^{+} (t,x)$ becomes
\begin{align*}
\widetilde{II}_{\infty}^{+}  (t,x) =&
\ const. (it)^{-M} \\
& \int_{\mathfrak{a}} d\lambda \ e^{it \psi (\lambda)} \Big\lbrace - \sum_{j=1}^{\ell} \frac{\partial}{\partial \lambda_{j}} \circ \frac{\lambda_{j}}{\widetilde{\psi}(\lambda)} \Big\rbrace^{M} \Big\lbrace \chi_{\infty}^{\rho}  ( t \lambda) \chi_{\infty}^{\rho}  (\lambda) \ \pi^2(\lambda) \ ( | \lambda |^2 + \widetilde{\rho}^2)^{ - \frac{\sigma}{2}} \Big\rbrace.
\end{align*}
Hence 
\begin{align*}
| \widetilde{II}_{\infty}^{+}  (t,x) | \lesssim& \ t^{-M} 
\int_{| \lambda | \ge | \rho | t^{-1}} d\lambda \ | \lambda |^{2 | \Sigma^{+} | -   \re \sigma- M} \\
\lesssim& \ t^{-M} \int_{r \ge  | \rho | t^{-1}} \frac{dr}{r} \ r^{2 | \Sigma^{+} | -   \re \sigma- M +\ell}.
\end{align*}
Therefore, for all $\sigma \in \mathbb{C}$ with $\re \sigma = \frac{d+1}{2}$, we have
\begin{align}\label{I tilde inf +}
| \widetilde{II}_{\infty}^{+}  (t,x) | \lesssim& \ t^{-M} t^{ - d + \re \sigma + M}
= t^{- \frac{d-1}{2}}
\end{align}
provided that $M > \frac{d-1}{2}$. From \eqref{I tilde inf -} and \eqref{I tilde inf +}, we deduce that
\begin{align*}
|\widetilde{\omega}_{t}^{\sigma, \infty}(x) | \lesssim \
t^{- \frac{d-1}{2}}  \int_{K} dk \ e^{ \langle \rho, A (xk) \rangle} = 
t^{- \frac{d-1}{2}}  \varphi_{0}(x), 
\end{align*}
for all $0 < t <1$ such that $\frac{|x|}{t} \le \frac{1}{2}$ and for all $\sigma \in \mathbb{C}$ with $\re \sigma = \frac{d+1}{2}$. Thus we have proved the last case, and the proof of \cref{estimate omega infinity} is complete.
\end{proof}

\begin{remark}
In the last proof, we have used the integral expression \eqref{spherical function noncompact} of the spherical function $\varphi_{\lambda}(x)$, which holds in general, and the particular expression \eqref{c function complex} of the Plancherel density $| \mathbf{c} (\lambda) |^{-2}$ when $G$ is complex. In general, $| \mathbf{c} (\lambda) |^{-2}$ is not a polynomial, nor even an inhomogeneous symbol of order $d-\ell$, which is known to be a major difficulty in higher rank analysis.
\end{remark}

\section{Dispersive estimates}\label{section dispersive}
In this section, we prove our main result about the $L^{q'} \rightarrow L^{q}$ estimates for the operator  $W_{t}^{\sigma} = \widetilde{D}^{-\sigma} e^{it \sqrt{- \Delta}}$. We introduce the following criterion based on the Kunze-Stein phenomenon, which is a straightforward generalization of \cite[Theorem 4.2]{APV2011} and which is crucial for our dispersive estimates.

\begin{lemma}\label{Kunze-Stein}
Let $\kappa$ be a $K$-bi-invariant measurable function on $G$. Then
\begin{align*}
\| \cdot * \kappa \|_{L^{q'}(X) \rightarrow L^{q}(X)} \lesssim \Big\lbrace \int_{G} dx \ \varphi_{0}(x) | \kappa(x) |^{\frac{q}{2}}  \Big\rbrace^{\frac{2}{q}}
\end{align*}
for any $q \in [ 2, + \infty )$.
In the limit case $q=\infty$,
\begin{align*}
\textstyle
    \|\cdot*\,\kappa\,\|_{L^{1}(\mathbb{X})
                    \rightarrow
                    L^{\infty}(\mathbb{X})}
    = \sup_{x\in{G}} |\kappa(x)|.
\end{align*}
\end{lemma}

\begin{proof}
For $s \in [2, + \infty)$, we define $\mathcal{A}_{s}$ as the space of all $K$-bi-invariant functions  on $G$ such that
\begin{align*}
\int_{G} dx \ \varphi_{0}(x) | \kappa (x) |^{s/2} < \infty.
\end{align*}
Given $\kappa$ in  $\mathcal{A}_{s}$, we set
\begin{align*}
\| \kappa \|_{\mathcal{A}_{s}} = \Big( \int_{G} dx \ \varphi_{0}(x) | \kappa (x) |^{s/2} \Big)^{2/s}.
\end{align*}
For $s = + \infty$, we denote by $\mathcal{A}_{\infty}$ the space of $L^{\infty} (G, dx)$ $K$-bi-invariant functions on $G$ and by $\| . \|_{\mathcal{A}_{\infty}}$ the $L^{\infty}$-norm. We show in the following that, for  any $q \in [ 2, + \infty ]$, we have
\begin{align*}
L^{q'} (G) * \mathcal{A}_{q} \subset L^{q} (G),
\end{align*}
which yields our lemma. 

The case where $q=2$ follows by Herz's criterion (see \cite{Cow1997}): 
\begin{align*}
\| \cdot  * \kappa \|_{L^2(G) \rightarrow L^2(G)} \le \int_{G} dx \ \varphi_{0}(x) | \kappa (x) | = \| \kappa \|_{\mathcal{A}_{2}}
\end{align*}
for every $\kappa$ in $\mathcal{A}_{2}$. When $q = + \infty$, taking $f$ in $L^{1} (G)$ and $\kappa$ in $\mathcal{A}_{\infty}$, we have that for every $x$ in $G$,
\begin{align*}
| f * \kappa (x) | \le \int_{G} dy \ |f(y)| |\kappa(y^{-1}x)| \le \| \kappa \|_{\mathcal{A}_{\infty}} \| f \|_{L^{1}(G)}.
\end{align*}
By interpolation between the cases $q=2$ and $q =\infty$, we obtain that
\begin{align*}
\big[ L^2(G), L^{1}(G) \big]_{\theta} * \big[ \mathcal{A}_{2} , \mathcal{A}_{\infty} \big]_{\theta} \subset  \big[ L^{2} (G), L^{\infty}(G) \big]_{\theta}
\end{align*}
with $\theta = 2/q$. On the one hand,
\begin{align*}
\big[ L^2(G), L^{1}(G) \big]_{\theta} = L^{q'} (G) 
\quad and \quad
 \big[ L^{2} (G), L^{\infty}(G) \big]_{\theta} = L^{q} (G).
\end{align*}
On the other hand, since
\begin{align*}
\big[ L^{1}(G, \varphi_{0} dx), L^{\infty}(G, \varphi_{0} dx) \big]_{\theta} 
= L^{q/2}(G, \varphi_{0} dx),
\end{align*}
we have $\big[ \mathcal{A}_{2} , \mathcal{A}_{\infty} \big]_{\theta} =  \mathcal{A}_{q}$. This concludes the proof of \cref{Kunze-Stein}.
\end{proof}

\subsection{Small time dispersive estimate}

\begin{theorem}\label{Small time dispersive estimate}
Assume that $0 < |t| <1$, $2 < q <  + \infty$ and $\sigma = (d+1)  ( \frac{1}{2} - \frac{1}{q} )$. Then
\begin{align*}
\| \widetilde{D}^{-\sigma} e^{it \sqrt{- \Delta}} \|_{L^{q'}(X) \rightarrow L^{q}(X)} \lesssim |t|^{(d-1) ( \frac{1}{2} - \frac{1}{q}) }.
\end{align*}
\end{theorem}

\begin{proof}
We divide the proof into two parts, corresponding to the kernel decomposition $\omega_{t}^{\sigma} =\omega_{t}^{\sigma,0} + \omega_{t}^{\sigma, \infty}$. According to \cref{Kunze-Stein}, we have
\begin{align*}
\| \cdot * \omega_{t}^{\sigma,0} \|_{L^{q'}(X) \rightarrow L^{q}(X)}
\le  \Big\lbrace \int_{G} dx \ \varphi_{0}(x) | \omega_{t}^{\sigma,0}(x) |^{\frac{q}{2}}  \Big\rbrace^{\frac{2}{q}}.
\end{align*}
By using the Cartan decomposition, together with the fact that $\delta(H) \lesssim e^{2 \left\langle {\rho, H} \right\rangle}$, we obtain
\begin{align*}
 \int_{G} dx \ \varphi_{0}(x) | \omega_{t}^{\sigma,0}(x) |^{\frac{q}{2}}
 \lesssim \int_{\mathfrak{a}^{+}} dH \ \varphi_{0}(H) | \omega_{t}^{\sigma,0}(H) |^{\frac{q}{2}} e^{2 \left\langle {\rho, H} \right\rangle}.
\end{align*}
According to \cref{estimate omega 0} and to the estimate \eqref{basic spherical function estimate} of the ground spherical function $\varphi_{0}$, we obtain next
\begin{align*}
\int_{\mathfrak{a}^{+}} dH \ \varphi_{0}(H) | \omega_{t}^{\sigma,0}(H) |^{\frac{q}{2}} e^{2 \left\langle {\rho, H} \right\rangle}
\lesssim \int_{\mathfrak{a}^{+}} dH \ (1 + |H| )^{\frac{q}{2} +1} e^{- (\frac{q}{2} - 1) \left\langle {\rho , H } \right\rangle} < + \infty
\end{align*}
for any $q \in (2, + \infty)$. Hence
\begin{align*}
\| \cdot * \omega_{t}^{\sigma,0} \|_{L^{q'}(X) \rightarrow L^{q}(X)} < + \infty.
\end{align*}
For the second part $\omega_{t}^{\sigma, \infty}$, we use an analytic interpolation between $L^2 \rightarrow L^2$ and $L^1 \rightarrow L^{\infty}$ estimates for the family of operators $\widetilde{W}_{t}^{\sigma, \infty}$ defined by \eqref{analytic family of operators} in the vertical strip $0 \le \re \sigma \le \frac{d+1}{2}$. When $\re \sigma = 0$, the spectral theorem yields
\begin{align*}
\| \widetilde{W}_{t}^{\sigma, \infty} \|_{L^2(X) \rightarrow L^2(X)}
\lesssim \| e^{it \sqrt{- \Delta}} \|_{L^2(X) \rightarrow L^2(X)} = 1.
\end{align*}
for all $t \in \mathbb{R}^{*}$. When $\re \sigma = \frac{d+1}{2}$, \cref{estimate omega infinity} gives
\begin{align*}
\| \widetilde{W}_{t}^{\sigma, \infty} \|_{L^1(X) \rightarrow L^{\infty}(X)}
\lesssim \| \widetilde{\omega}_{t}^{\sigma, \infty} \|_{L^{\infty}(X)} \lesssim t^{- \frac{d-1}{2}}.
\end{align*}
By applying Stein's interpolation theorem for an analytic family of operators, we obtain
\begin{align*}
\| \widetilde{W}_{t}^{\frac{d+1}{2}(1- \theta), \infty} \|_{L^{q'}(X) \rightarrow L^{q}(X)} \lesssim t^{- \frac{d-1}{2} (1- \theta)},
\end{align*}
with $\theta = \frac{2}{q}$. In conclusion,
\begin{align*}
\| W_{t}^{\sigma} \|_{L^{q'}(X) \rightarrow L^{q}(X)} \lesssim |t|^{-(d-1) ( \frac{1}{2} - \frac{1}{q}) },
\end{align*}
for $0 < |t| < 1$, $2 < q < + \infty$ and $\sigma = (d+1) (\frac{1}{2} - \frac{1}{q})$.
\end{proof}
\subsection{Large time dispersive estimate}
\begin{theorem}\label{Large time dispersive estimate}
Assume that $|t| \ge 1$, $2 < q <  + \infty$ and $\sigma = (d+1)  ( \frac{1}{2} - \frac{1}{q} )$. Then
\begin{align*}
\| \widetilde{D}^{-\sigma} e^{it \sqrt{- \Delta}} \|_{L^{q'}(X) \rightarrow L^{q}(X)} \lesssim |t|^{- \frac{d}{2}}.
\end{align*}
\end{theorem}

\begin{proof}
We divide the proof into three parts, corresponding to the kernel decomposition
\begin{align*}
\omega_{t}^{\sigma} = \mathds{1}_{B(0, \frac{|t|}{2})} \omega_{t}^{\sigma,0} + \mathds{1}_{X \backslash B(0, \frac{|t|}{2})} \omega_{t}^{\sigma,0} +  \omega_{t}^{\sigma, \infty}.
\end{align*}
By using \cref{Kunze-Stein} and \cref{estimate omega 0}, we have
\begin{align*}
\| \cdot * \lbrace \mathds{1}_{B(0, \frac{|t|}{2})} \omega_{t}^{\sigma,0} & \rbrace \|_{L^{q'}(X) \rightarrow L^{q}(X)} \\
\lesssim& \Big\lbrace \int_{\lbrace H \in \mathfrak{a}^{+}: 0 < |H| < \frac{|t|}{2} \rbrace} dH \ \varphi_{0}(H) | \omega_{t}^{\sigma,0}(H) |^{\frac{q}{2}} e^{2 \left\langle {\rho, H} \right\rangle} \Big\rbrace^{\frac{2}{q}} \\
\lesssim& t^{-\frac{d}{2}}  \underbrace{ \Big\lbrace \int_{|H| < \frac{|t|}{2}} dH \ (1+|H|)^{\frac{q}{2} (d+1)} e^{- (\frac{q}{2} - 1) \left\langle {\rho , H } \right\rangle} \Big\rbrace^{\frac{2}{q}} }_{< + \infty}
\end{align*}
and
\begin{align*}
\| \cdot * \lbrace \mathds{1}_{X \backslash B(0, \frac{|t|}{2})} \omega_{t}^{\sigma,0} & \rbrace \|_{L^{q'}(X) \rightarrow L^{q}(X)} \\
\lesssim&  \Big\lbrace \int_{|H| \ge \frac{|t|}{2}} dH \ (1+|H|)^{\frac{q}{2} (d+1)} e^{- (\frac{q}{2} - 1) \left\langle {\rho , H } \right\rangle} \Big\rbrace^{\frac{2}{q}} ,
\end{align*}
which is $O(|t|^{-N})$, for any $N \in \mathbb{N}$. Instead of $\omega_{t}^{\sigma, \infty}$, we consider again the kernel $\widetilde{\omega}_{t}^{\sigma, \infty}$. The associated operators satisfy
\begin{align*}
\| \widetilde{W}_{t}^{\sigma, \infty} \|_{L^1(X) \rightarrow L^{\infty}(X)}
\lesssim t^{-N} \qquad \forall N \in \mathbb{N},
\end{align*}
when $\re \sigma = \frac{d+1}{2}$, according to \cref{estimate omega infinity}. By applying Stein's interpolation theorem for an analytic family of operators and by summing up these estimates, we obtain
\begin{align*}
\| W_{t}^{\sigma} \|_{L^{q'}(X) \rightarrow L^{q}(X)} \lesssim |t|^{- \frac{d}{2} },
\end{align*}
for $ |t| \ge 1$, $2 < q < + \infty$ and $\sigma = (d+1) (\frac{1}{2} - \frac{1}{q})$.
\end{proof}

\begin{remark}
The standard $TT^{*}$ method used to prove the Strichartz inequality breaks down in the critical case. The dyadic decomposition method carried out in \cite{KeTa1998} takes care of the endpoints, but it requires a stronger dispersive property than \cref{Small time dispersive estimate} in small time, namely our main theorem.
\end{remark}

\begin{proof}[Proof of \cref{dispersive estimate}]
It follows from \cref{Small time dispersive estimate}, \cref{Large time dispersive estimate} and the $L^q$ kernel estimate
\begin{align*}
\| \omega_{t}^{\sigma} \|_{L^{q}(X)}
\le \| \omega_{t}^{\sigma,0} \|_{L^{q}(X)} + \| \widetilde{\omega}_{t}^{\sigma, \infty} \|_{L^{q}(X)}
\lesssim
\begin{cases}
|t|^{- \frac{d-1}{2}}, \qquad & \textnormal{if}\quad 0 < |t| < 1, \\
|t|^{- \frac{d}{2}}, \qquad & \textnormal{if}\quad |t| \ge 1,
\end{cases}
\end{align*}
for $2 < q < + \infty$ and $\sigma = \frac{d+1}{2} ( \frac{1}{2} - \frac{1}{q} )$. Indeed, by using standard interpolation arguments between
\begin{align*}
\begin{cases}
\| W_{t}^{\sigma} \|_{L^{1}(X) \rightarrow L^{q}(X)} 
\lesssim \| \omega_{t}^{\sigma} \|_{L^{q}(X)} \lesssim t^{- \frac{d-1}{2}}, \\
\| W_{t}^{\sigma} \|_{L^{q'}(X) \rightarrow L^{\infty}(X)} 
\lesssim \| \omega_{t}^{\sigma} \|_{L^{q}(X)} \lesssim t^{- \frac{d-1}{2}}, \\
\| W_{t}^{\sigma} \|_{L^{2}(X) \rightarrow L^{2}(X)} = 1,
\end{cases}
\end{align*}
for small time and
\begin{align*}
\begin{cases}
\| W_{t}^{\sigma} \|_{L^{1}(X) \rightarrow L^{q}(X)} 
\lesssim \| \omega_{t}^{\sigma} \|_{L^{q}(X)} \lesssim t^{- \frac{d}{2}}, \\
\| W_{t}^{\sigma} \|_{L^{q'}(X) \rightarrow L^{\infty}(X)} 
\lesssim \| \omega_{t}^{\sigma} \|_{L^{q}(X)} \lesssim t^{- \frac{d}{2}}, \\
\| W_{t}^{\sigma} \|_{L^{q'}(X) \rightarrow L^{q}(X)} \lesssim |t|^{- \frac{d}{2}},
\end{cases}
\end{align*}
for large time, we conclude.
\end{proof}

\section{Strichartz inequality and applications}\label{section applcations}
Let $\sigma \in \mathbb{R}$ and $1<q< \infty$. Recall that the Sobolev  space $H^{\sigma,q}(X)$ is the image of $L^q(X)$ under the operator $(- \Delta)^{- \frac{\sigma}{2}}$, equipped with the norm
\begin{align*}
\| f \|_{H^{\sigma,q}(X)} = \| (- \Delta)^{ \frac{\sigma}{2}} f \|_{L^q(X)}.
\end{align*}
If  $\sigma=N$ is a nonnegative integer, then $H^{\sigma,q}(X)$ coincides with the classical Sobolev space
\begin{align*}
W^{N,q} (X) = \lbrace f \in L^q(X) \ | \ \nabla^{j} f \in L^q(X), \ \forall 1 \le j \le N \rbrace,
\end{align*}
defined by means of covariant derivatives. The following Sobolev embedding theorem is used in next subsection.
\begin{theorem}
Let $1< q_1,q_2 < \infty$ and $\sigma_1, \sigma_2 \in \mathbb{R}$ such that $\sigma_1 - \sigma_2 \ge \frac{d}{q_1} - \frac{d}{q_2} \ge 0$. Then
\begin{align}\label{embedding}
H^{\sigma_1,q_1}(X) \subset H^{\sigma_2,q_2}(X).
\end{align}
\end{theorem}
We refer to \cite{Tri1992} for more details about function spaces on Riemannian manifolds. Let us state next the Strichartz inequality and some applications. The proofs are adapted straightforwardly from \cite{AnPi2014, APV2012} and are therefore omitted.
\subsection{Strichartz inequality}\label{section Strichartz}
Recall the linear inhomogeneous Klein-Gordon equation on $X$:
\begin{align}\label{Cauchy problem}
\begin{cases}
\partial_{t}^2 u(t,x) - \Delta u(t,x) = F(t,x), \\
u(0,x) =f(x) , \ \partial_{t}|_{t=0} u(t,x) =g(x).
\end{cases}
\end{align}
whose solution is given by Duhamel's formula:
\begin{align*}
u(t,x) 
= ( \cos t \sqrt{- \Delta} ) f(x)
+ \frac{\sin t\sqrt{- \Delta}}{\sqrt{- \Delta}} g(x)
+ \int_{0}^{t} \frac{\sin (t-s) \sqrt{- \Delta}}{\sqrt{- \Delta}} F(s,x)ds.
\end{align*}

We consider first the case $d \ge 4$ and discuss the $3$-dimensional case in the final remark. Recall that a couple $(p,q)$ is called admissible if $( \frac{1}{p}, \frac{1}{q} )$ belongs to the triangle
\begin{align*}
\textstyle
\Big\lbrace \Big( \frac{1}{p}, \frac{1}{q} \Big) \in \Big( 0, \frac{1}{2} \Big] \times \Big( 0, \frac{1}{2} \Big) \ \Big| \ \frac{1}{p} \ge \frac{d-1}{2} \Big( \frac{1}{2} - \frac{1}{q} \Big) \Big\rbrace 
\bigcup \Big\lbrace  \Big( 0, \frac{1}{2} \Big) \Big\rbrace.
\end{align*}
\vspace{-0.5cm}
\begin{figure}[!h]\label{n >4}
\centering
\begin{tikzpicture}[scale=0.47][line cap=round,line join=round,>=triangle 45,x=1.0cm,y=1.0cm]
\draw[->,color=black] (0.,0.) -- (8.805004895364174,0.);
\foreach \x in {,2.,4.,6.,8.,10.}
\draw[shift={(\x,0)},color=black] (0pt,-2pt);
\draw[color=black] (8,0.08808504628984362) node [anchor=south west] {\large $\frac{1}{p}$};
\draw[->,color=black] (0.,0.) -- (0.,8);
\foreach \y in {,2.,4.,6.,8.}
\draw[shift={(0,\y)},color=black] (-2pt,0pt);
\draw[color=black] (0.11010630786230378,8) node [anchor=west] {\large $\frac{1}{q}$};
\clip(-5,-2) rectangle (11.805004895364174,9.404473957611293);
\fill[line width=2.pt,color=ffqqqq,fill=ffqqqq,fill opacity=0.15000000596046448] (0.45855683542994374,5.928142080369191) -- (5.85410415683891,5.8847171522290775) -- (5.864960388873938,4.136863794589543) -- cycle;
\draw [line width=1.pt,dash pattern=on 5pt off 5pt,color=ffqqqq] (0.,6.)-- (6.,6.);
\draw [line width=1.pt,color=ffqqqq] (6.,6.)-- (6.,4.);
\draw [line width=1.pt,color=ffqqqq] (6.,4.)-- (0.,6.);
\draw [line width=0.8pt] (6.,4.)-- (6.,0.);
\draw [->,line width=0.5pt,color=ffqqqq] (1.8569381602310457,-0.8701903243695805) -- (2.891937454136701,4.811295161325333);
\begin{scriptsize}
\draw [fill=ffqqqq] (0.,6.) circle (5pt);
\draw[color=black] (-0.5709441083587729,6) node {\large $\frac{1}{2}$};
\draw[color=black] (-1.5,4) node {\large $\frac{1}{2}- \frac{1}{d-1}$};
\draw[color=black] (-0.5,-0.2) node {\large $0$};
\draw[color=black] (6.2,-0.7) node {\large $\frac{1}{2}$};
\draw[color=ffqqqq] (1.8,-1.3) node {\large $\frac{1}{p} = \frac{d-1}{2} \left( \frac{1}{2} - \frac{1}{q} \right)$};
\draw [color=ffqqqq] (6.,6.) circle (5pt);
\draw [fill=ffqqqq] (6.,4.) circle (5pt);
\end{scriptsize}
\end{tikzpicture}
\caption{Admissibility in dimension $d \ge 4$.}
\end{figure}
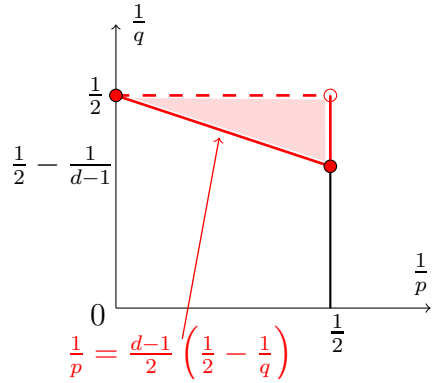

\begin{theorem}\label{thm strichartz}
Let $(p,q)$ and $(\tilde{p}, \tilde{q})$ be two admissible couples, and let
\begin{align*}
\textstyle
\sigma \ge \frac{d+1}{2} \left( \frac{1}{2} - \frac{1}{q} \right) \quad and \quad \widetilde{\sigma} \ge \frac{d+1}{2} \left( \frac{1}{2} - \frac{1}{\tilde{q}} \right).
\end{align*}
Then all solutions $u$ to the Cauchy problem \eqref{Cauchy problem} satisfy the  following Strichartz inequality:
\begin{align}\label{Strichartz}
\left\| \nabla_{\mathbb{R} \times X} u \right\|_{L^p(I; H^{-\sigma,q}(X))} \lesssim \left\| f \right\|_{H^1(X)} + \left\| g \right\|_{L^2(X)} + \left\| F \right\|_{L^{\tilde{p}'}(I; H^{\tilde{\sigma},\tilde{q}'}(X))}.
\end{align}
\end{theorem}

The admissible range in \eqref{Strichartz} can be widen by using the Sobolev embedding theorem.

\begin{corollary}\label{corollary strichartz}
Assume that $(p,q)$ and $(\tilde{p}, \tilde{q})$ are two couples corresponding to the square 
\begin{align*}
\Big[ 0, \frac{1}{2} \Big] \times \Big(0, \frac{1}{2} \Big) \bigcup \Big\lbrace \Big( 0, \frac{1}{2} \Big) \Big\rbrace,
\end{align*}
see Fig.(A) in the following.
Let $\sigma, \tilde{\sigma} \in \mathbb{R}$ such that $\sigma \ge \sigma(p,q)$, where
\begin{align*}
\textstyle
\sigma(p,q) =& \frac{d+1}{2} \Big( \frac{1}{2} - \frac{1}{q} \Big) + \max \Big\lbrace 0, \frac{d-1}{2} \Big( \frac{1}{2} - \frac{1}{q} \Big) - \frac{1}{p} \Big\rbrace,
\end{align*}
and similarly $\tilde{\sigma} \ge \sigma(\tilde{p}, \tilde{q})$. Then the Strichartz inequality \eqref{Strichartz} holds for all solutions to the Cauchy problem \eqref{Cauchy problem}.
\end{corollary}

\vspace{-0.46cm}
\begin{figure}[!h]
\begin{subfigure}[t]{0.5\textwidth}
\centering
\begin{tikzpicture}[scale=0.5][line cap=round,line join=round,>=triangle 45,x=1.0cm,y=1.0cm]
\draw[->,color=black] (6.,0.) -- (8,0.);
\foreach \x in {,1.,2.,3.,4.,5.,6.,7.,8.,9.,10.,11.}
\draw[shift={(\x,0)},color=black] (0pt,-2pt);
\draw[color=black] (8,0.1) node [anchor=south west] {\large $\frac{1}{p}$};
\draw[->,color=black] (0.,6.) -- (0.,8);
\foreach \y in {,1.,2.,3.,4.,5.,6.,7.}
\draw[shift={(0,\y)},color=black] (-2pt,0pt);
\draw[color=black] (0.1,8.1) node [anchor=west] {\large $\frac{1}{q}$};
\clip(-4,-2) rectangle (11.868487651785372,7.652018714581741);
\fill[line width=2.pt,color=ffqqqq,fill=ffqqqq,fill opacity=0.15000000596046448] (0.45855683542994374,5.928142080369191) -- (5.85410415683891,5.8847171522290775) -- (5.864960388873938,4.136863794589543) -- cycle;
\fill[line width=2.pt,color=qqzzff,fill=qqzzff,fill opacity=0.10000000149011612] (0.13636385492209063,5.773698170331237) -- (5.837359246154291,3.8701454719028483) -- (5.876010062366035,0.10169089125781497) -- (0.16535196708089844,0.12101629936368694) -- cycle;
\draw [line width=1.pt,dash pattern=on 3pt off 3pt,color=ffqqqq] (0.,6.)-- (6.,6.);
\draw [line width=1.pt,color=ffqqqq] (6.,6.)-- (6.,4.);
\draw [line width=1.pt,color=ffqqqq] (6.,4.)-- (0.,6.);
\draw [line width=1.pt,color=qqzzff] (0.,6.)-- (0.,0.);
\draw [line width=1.pt,color=qqzzff] (6.,4.)-- (6.,0.);
\draw [line width=1.pt,dash pattern=on 3pt off 3pt,color=qqzzff] (6.,0.)-- (0.,0.);
\begin{scriptsize}
\draw [fill=ffqqqq] (0.,6.) circle (5pt);
\draw[color=black] (-0.5709441083587729,6) node {\large $\frac{1}{2}$};
\draw [color=ffqqqq] (6.,6.) circle (5pt);
\draw [fill=ffqqqq] (6.,4.) circle (5pt);
\draw [color=qqzzff] (0.,0.) circle (5pt);
\draw[color=black] (-0.5,-0.2) node {\large $0$};
\draw [color=qqzzff] (6.,0.) circle (5pt);
\draw[color=black] (6,-1) node {\large $\frac{1}{2}$};
\draw[color=black] (-1.5,4) node {\large $\frac{1}{2}- \frac{1}{d-1}$};
\end{scriptsize}
\end{tikzpicture}
\vspace{-0.5cm}
\subcaption{Case $d \ge 4$.}
\end{subfigure}
~ 
\begin{subfigure}[t]{0.5\textwidth}
\centering
\begin{tikzpicture}[scale=0.5][line cap=round,line join=round,>=triangle 45,x=1.0cm,y=1.0cm]
\draw[->,color=black] (6.,0.) -- (8,0.);
\foreach \x in {,1.,2.,3.,4.,5.,6.,7.,8.,9.,10.,11.}
\draw[shift={(\x,0)},color=black] (0pt,-2pt);
\draw[color=black] (8,0.1) node [anchor=south west] {\large $\frac{1}{p}$};
\draw[->,color=black] (0.,0.) -- (0.,8);
\foreach \y in {,1.,2.,3.,4.,5.,6.,7.}
\draw[shift={(0,\y)},color=black] (-2pt,0pt);
\draw[color=black] (0.1,8.1) node [anchor=west] {\large $\frac{1}{q}$};
\clip(-4,-2) rectangle (11.868487651785372,7.652018714581741);
\fill[line width=2.pt,color=ffqqqq,fill=ffqqqq,fill opacity=0.15000000596046448] (0.23947996857607806,5.9156932292089195) -- (5.85410415683891,5.8847171522290775) -- (5.896883481046051,0.25828971673893225) -- cycle;
\fill[line width=2.pt,color=qqzzff,fill=qqzzff,fill opacity=0.10000000149011612] (0.08350079729924402,5.793380424802772) -- (0.11455249049619687,0.12106762660967957) -- (5.704457330452569,0.11429587297972885) -- cycle;
\draw [line width=1.pt,dash pattern=on 3pt off 3pt,color=ffqqqq] (0.,6.)-- (6.,6.);
\draw [line width=1.pt,color=ffqqqq] (6.,6.)-- (6.,0.);
\draw [line width=1.pt,color=qqzzff] (0.,6.)-- (0.,0.);
\draw [line width=1.pt,color=red] (0.,6.)-- (6.,0.);
\draw [line width=1.pt,dash pattern=on 3pt off 3pt,color=qqzzff] (0.,0.)-- (6.,0.);
\begin{scriptsize}
\draw [fill=ffqqqq] (0.,6.) circle (5pt);
\draw[color=black] (-0.5709441083587729,6) node {\large $\frac{1}{2}$};
\draw [color=ffqqqq] (6.,6.) circle (5pt);
\draw [color=ffqqqq] (6.,0.) circle (5pt);
\draw [color=qqzzff] (0.,0.) circle (5pt);
\draw[color=black] (-0.5,-0.2) node {\large $0$};
\draw[color=black] (6,-1) node {\large $\frac{1}{2}$};
\end{scriptsize}
\end{tikzpicture}
\vspace{-0.5cm}
\subcaption{Case $d = 3$.}
\end{subfigure}
\end{figure}

\begin{remark}
\cref{thm strichartz} and \cref{corollary strichartz} still hold true in lower dimension $d=3$ with similar proofs. But the endpoint $(\frac{1}{2}, \frac{1}{2}- \frac{1}{d-1} )$ is excluded from the admissible triangle in this case, see  Fig.(B).
\end{remark}

\subsection{Global well-posedness in $L^{p} \left( \mathbb{R}, L^q (X) \right)$}\label{section GWP} 
We refer to \cite{AnPi2014,APV2012} for more detailed proofs of the following well-posedness results. By using the classical fixed point scheme with the previous Strichartz inequality, one obtains the global well-posedness for the semilinear equation
\begin{align}\label{semilinear equation}
\begin{cases}
\partial_{t}^2 u(t,x) - \Delta u(t,x) = F(u(t,x)), \\
u(0,x) =f(x) , \ \partial_{t}|_{t=0} u(t,x) =g(x).
\end{cases}
\end{align}
on $X$ with power-like nonlinearities $F$ satisfying
\begin{align*}
|F(u)| \le C |u|^{\gamma}, \quad |F(u)-F(v)| \le C \left(|u|^{\gamma-1} + |v|^{\gamma -1} \right) |u-v|, \quad \gamma >1.
\end{align*}
and small initial data $f$ and $g$. For every $d \ge 3$, consider the following powers
\begin{align*}
\gamma_1 = 1+ \frac{3}{d}, \ \gamma_2 = 1 + \frac{2}{\frac{d-1}{2}+ \frac{2}{d-1}}, \ \gamma_c = 1+ \frac{4}{d-1}, \\
\gamma_3 = 
\begin{cases}
\frac{\frac{d+6}{2} + \frac{2}{d-1} + \sqrt{4d + \left( \frac{6-d}{2}+\frac{2}{d-1} \right)^2}}{d} \quad &if \ d \le 5, \\
1 + \frac{2}{\frac{d-1}{2} - \frac{1}{d-1}}\quad &if \ d \ge 6,
\end{cases} \\
\gamma_4 = 
\begin{cases}
1 + \frac{4}{d-2} \quad &if \ d \le 5, \\
\frac{d-1}{2} + \frac{3}{d+1} - \sqrt{\left( \frac{d-3}{2} + \frac{3}{d+1} \right)^2 - 4 \frac{d-1}{d+1}}  &if \ d \ge 6,
\end{cases}
\end{align*}
and the following curves
\begin{align*}
\sigma_1(\gamma) = \frac{d+1}{4} - \frac{(d+1)(d+5)}{8d} \frac{1}{\gamma - \frac{d+1}{2d}}, \\
\sigma_2(\gamma) = \frac{d+1}{4} - \frac{1}{\gamma -1}, \quad 
\sigma_3(\gamma) = \frac{d}{2} - \frac{2}{\gamma -1}.
\end{align*}
Denote by $0^{+}$ any small positive constant. In dimension $d \ge 3$, the equation \eqref{semilinear equation} is globally well-posed for small initial data in $H^{\sigma,2}(X) \times H^{\sigma - 1,2}(X)$ provided that
\begin{align*}
\begin{cases}
\sigma = 0^{+}, \ &if \ 1< \gamma \le \gamma_1, \\
\sigma = \sigma_1 (\gamma), \ &if \ \gamma_1< \gamma \le \gamma_2, \\
\sigma = \sigma_2 (\gamma), \ &if \ \gamma_2< \gamma \le \gamma_c, \\
\sigma = \sigma_3 (\gamma), \ &if \ \gamma_c< \gamma \le \gamma_4,
\end{cases}
\end{align*}

Observe that one obtains the same global well-posedness results on noncompact symmetric spaces of arbitrary rank with $G$ complex as on real hyperbolic spaces, without further assumptions.
\section{Further results on locally symmetric spaces}\label{LSS}
Let $\Gamma$ be a discrete torsion-free subgroup of $G$. The locally symmetric space $\Gamma \backslash X$, equipped with the Riemannian structure inherited from $X$ becomes a Riemannian manifold. Consider the Poincaré series
\begin{align*}
P(s;x,y) = \sum_{\gamma \in \Gamma} e^{-s d(x, \gamma y)}, \ s > 0 ,  \ x,y \in X,
\end{align*}
and denote by $\delta(\Gamma) = \inf \lbrace s > 0 \ | \ P(s;x,y) < + \infty \rbrace $ its critical exponent. In \cite{Zha2020}, the author has studied the wave equation and has obtained similar Strichartz inequality and global well-posedness results as in \cref{section applcations}, in the case where $\Gamma \backslash X$ is a rank one locally symmetric space such that $\Gamma$ is convex cocompact and $\delta(\Gamma) < |\rho|$. \\

Recall that,in the rank one setting, $\Gamma$ is called convex cocompact if the quotient group $\Gamma \backslash \conv (\Lambda_{\Gamma})$ is compact, where $\conv (\Lambda_{\Gamma})$ is the convex hull of the limit set $\Lambda_{\Gamma}$ of $\Gamma$. While this notion yields many interesting examples in rank one, it is known to yield a rather limited class in higher rank (see \cite{KlLe2006} and \cite{Qui2005} for more details).\\

However, thanks to our wave kernel estimates \cref{estimate omega 0} and \cref{estimate omega infinity}, we can study along the lines of \cite{Zha2020} the wave equation on higher rank noncompact locally symmetric spaces, under slightly different assumptions:
\begin{enumerate}[leftmargin=1cm]
\item $G$ is complex,
\item $\delta(\Gamma) < |\rho|$,
\item there exists $C > 0$ such that for all $x,y \in X$, $P(s;x,y) \le C P(s; \mathbf{0}, \mathbf{0})$, where $\mathbf{0} =eK$ denotes the origin of $X$.
\end{enumerate}

The first assumption ensures sharp wave kernel estimates on $X$, from which we can deduce wave kernel estimates on $\Gamma \backslash X$. Notice that such information is still lacking for $G$ real. The second assumption plays the same role as in rank one. On the one hand, it ensures that the wave kernel on $\Gamma \backslash X$ is well defined. On the other hand, there is a $L^2$ Kunze-Stein phenomenon under this assumption. In order to get the desired dispersive estimates on $\Gamma \backslash X$, a uniform upper bound of the Poincaré series is required. Notice that we could deduce the last condition $(3)$ from the convex cocompactness of $\Gamma$ in rank one. We refer to \cite{Zha2020} and the references therein for more details about wave type equations on locally symmetric spaces.


\printbibliography

\vspace{15pt}
\address{
    \noindent\textsc{Hong-Wei Zhang:}
    \href{mailto:hongwei.zhang@ugent.be}
    {hongwei.zhang@ugent.be}\\[5pt]
    Institut Denis Poisson,
    Universit\'e d'Orl\'eans, Universit\'e de Tours \& CNRS, 
    Orléans, France\\[5pt]
    Department of Mathematics:
    Analysis, Logic and Discrete Mathematics\\
    Ghent University, Ghent, Belgium
}

\end{document}